\documentclass[a4paper, marginparwidth=15mm]{amsart}

\usepackage{amsthm}
\usepackage{amssymb}
\usepackage{dsfont}
\usepackage{enumitem}
\usepackage{stmaryrd}
\usepackage{colonequals}
\usepackage[mathcal]{euscript}
\usepackage{tikz-cd} 
\usepackage[all]{xy}
\SelectTips{cm}{}
\usepackage{microtype}
\usepackage{mathtools}

\usepackage[top=35mm, bottom=35mm, outer=30mm, inner=30mm, marginparwidth=20mm]{geometry}

\usepackage{hyperref}
\hypersetup{%
  bookmarksnumbered=true,%
  colorlinks=true,%
  linkcolor=[RGB]{0, 0, 153},
  citecolor=[RGB]{0,153,51},      
  urlcolor=[RGB]{153,0,0},
  filecolor=blue,%
  menucolor=blue,%
  bookmarksopen=true,%
  bookmarksdepth=2,%
  pageanchor=true,
  breaklinks=true,
  pdfusetitle=true}

\usepackage[nameinlink,capitalize]{cleveref}
\crefname{equation}{}{}
\usepackage[backend=bibtex, style=alphabetic]{biblatex}
\renewbibmacro{in:}{}
\makeatletter
\@namedef{subjclassname@2020}{\textup{2020} Mathematics Subject Classification}
\makeatother

\numberwithin{equation}{section}

\theoremstyle{plain}
\newtheorem{theorem}[equation]{Theorem}
\newtheorem{proposition}[equation]{Proposition}
\newtheorem{lemma}[equation]{Lemma} 
\newtheorem{corollary}[equation]{Corollary}
\newtheorem{conjecture}[equation]{Conjecture}

\theoremstyle{definition}
\newtheorem{definition}[equation]{Definition}
\newtheorem{example}[equation]{Example}

\newtheorem*{claim}{Claim}

\theoremstyle{remark}
\newtheorem{remark}[equation]{Remark} 

\newtheorem*{ack}{Acknowledgements}
\newtheorem*{conventions}{Conventions}

\newcommand*{\intref}[2]{\def\tmp{#1}\ifx\tmp\empty\hyperref[#2]{\ref*{#2}}\else\hyperref[#2]{#1~\ref*{#2}}\fi}

\newcommand{\Aut}{\operatorname{Aut}}
\newcommand{\Bl}{\operatorname{Bl}}
\newcommand{\cc}{\operatorname{c}}
\newcommand{\CH}{\operatorname{CH}}
\newcommand{\ch}{\operatorname{ch}}

\renewcommand{\dim}{\operatorname{dim}}

\newcommand{\Hom}{\operatorname{Hom}}
\newcommand{\id}{\operatorname{id}}
\renewcommand{\Im}{\operatorname{Im}}

\renewcommand{\mod}{\operatorname{mod}}

\newcommand{\NS}{\operatorname{NS}}
\newcommand{\Orth}{\operatorname{O}}

\newcommand{\Pic}{\operatorname{Pic}}

\newcommand{\rk}{\operatorname{rk}}

\newcommand{\RHom}{\operatorname{{\mathbf R}Hom}}

\newcommand{\mcF}{\mathcal{F}}

\newcommand{\mcO}{\mathcal{O}}

\newcommand{\mcT}{\mathcal{T}}

\newcommand{\sfp}{\mathsf p} 
\newcommand{\sfq}{\mathsf q} 
\newcommand{\sfr}{\mathsf r}

\newcommand{\sfD}{\mathsf D}

\newcommand{\sfK}{\mathsf K} 
\newcommand{\sfL}{\mathsf L} 
\newcommand{\sfM}{\mathsf M}

\newcommand{\sfR}{\mathsf R}

\newcommand{\bbC}{\mathbb C}

\newcommand{\bbP}{\mathbb P}
\newcommand{\bbQ}{\mathbb Q} 
\newcommand{\bbZ}{\mathbb Z}

\title[Mutations of Numerically Exceptional Collections on Surfaces]{Mutations of Numerically Exceptional Collections on Surfaces}

\author[Johannes~Krah]{Johannes Krah}
\address{Fakult\"at f\"ur Mathematik, Universit\"at Bielefeld, D-33501 Bielefeld, Germany
}
\email{jkrah@math.uni-bielefeld.de}

\thanks{The research was funded by the Deutsche Forschungsgemeinschaft (DFG, German Research Foundation) -- SFB-TRR 358/1 2023 -- 491392403}

\begin{document}

\begin{abstract}
A conjecture of Bondal--Polishchuk states that, in particular for the bounded derived category of coherent sheaves on a smooth projective variety, the action of the braid group on full exceptional collections is transitive up to shifts.
We show that the braid group acts transitively on the set of maximal numerically exceptional collections on rational surfaces up to isometries of the Picard lattice and twists with line bundles. Considering the blow-up of the projective plane in up to $9$ points in very general position, these results lift to the derived category. More precisely, we prove that, under these assumptions, a maximal numerically exceptional collection consisting of line bundles is a full exceptional collection and any two of them are related by a sequence of mutations and shifts. The former extends a result of Elagin--Lunts and the latter a result of Kuleshov--Orlov, both concerning del Pezzo surfaces.
In contrast, we show in concomitant work \cite{krah_upcoming_work} that the blow-up of the projective plane in $10$ points in general position admits a non-full exceptional collection of maximal length consisting of line bundles.
\end{abstract}



\maketitle
\section{Introduction}\label{se:introduction}
Any smooth projective rational surface over an algebraically closed field admits a full exceptional collection by Orlov's projective bundle and blow-up formulae \cite{orlov_projective_bundles_monoidal_transformations_and_derived_categories_of_coherent_sheaves}, however a classification of exceptional collections on a given surface is widely open.
To construct new exceptional collections from old ones, a key tool are so-called \emph{mutations} of exceptional pairs, see \cref{sec:mutations}; these give rise to an action of the braid group in $n$ strands on the set of exceptional collections of length $n$ on such a surface.
Bondal and Polishchuk conjectured in more generality:
\begin{conjecture}[{\cite[Conj.~2.2]{bondal_polishchuk_homological_properties_of_associative_algebras_the_method_of_helices}}]\label{conj:bondal_polishchuck}
	Let $\mcT$ be a triangulated category which admits a full exceptional collection $\mcT=\langle E_1, \dots, E_n \rangle$. Then any other full exceptional collection of $\mcT$ can be constructed from $\langle E_1, \dots, E_n \rangle$ by a sequence of mutations and shifts.
\end{conjecture}
Recently, this conjecture was proven to be false \cite{chang_haiden_schroll_braid_group_actions_on_branched_coverings_and_full_exceptional_sequences} and a counterexample is given by a Fukaya category of a certain smooth two-dimensional real manifold.
To our knowledge, the conjecture still remains open for triangulated categories $\mcT = \sfD^b(\mathrm{Coh}(X))$, where $X$ is a smooth projective variety.

This paper is concerned with the question of classifying (numerically) exceptional collections on a given algebraic surface.
Exceptional collections on rational surfaces have been previously studied in \cite{hille_perling_exceptional_sequences_of_invertible_sheaves_on_rational_surfaces} and \cite{perling_combinatorial_aspects_of_exceptional_sequences_on_rationl_surfaces} via considering their associated toric surfaces. A classification of surfaces admitting a \emph{numerically exceptional collection of maximal length} was carried out in \cite{vial_exceptional_collections_and_the_neron_severi_lattice_for_surfaces}.
\cref{conj:bondal_polishchuck} was first verified in the cases $\mcT = \sfD^b(X)$, where $X$ is either $\bbP^2$ or $\bbP^1\times \bbP^1$. The case where $X$ is a del Pezzo surfaces is treated in \cite{kuleshov_orlov_exceptional_sheaves_on_del_pezzo_surfaces}. In \cite{kuleshov_exceptional_and_rigid_sheaves_on_surfaces_with_anticaonical_class_without_base_components}, similar results for surfaces with basepoint-free anticanonical class were obtained. A full discussion of exceptional collections on the Hirzebruch surface $\Sigma_2$ was worked out in \cite{ishi_okawa_uehara_exceptional_collections_on_sigma_2} and \cref{conj:bondal_polishchuck} was settled for $\sfD^b(\Sigma_2)$.

In the first part of the paper, we consider the images of exceptional collections in $\sfK_0^{\mathrm{num}}(X)$ instead of the objects in the derived category itself. These so-called \emph{numerically exceptional collections} on a surface $X$ with $\chi(\mcO_X)=1$ have been previously investigated by Perling and Vial \cite{perling_combinatorial_aspects_of_exceptional_sequences_on_rationl_surfaces, vial_exceptional_collections_and_the_neron_severi_lattice_for_surfaces}. Their lattice-theoretic arguments have been reworked by Kuznetsov in the abstract setting of \emph{surface-like pseudolattices}, introduced in \cite{kuznetsov_exceptional_collections_in_surface_like_categories}. Independently, a similar notion of a \emph{surface-type Serre lattice} was developed in \cite{de_thanhoffer_de_volcsey_van_den_berg_on_an_analogue_of_the_markov_equation_for_exceptional_collections_of_length_4}. In \cref{se:basics} we unify both formalisms in order to prove in \cref{se:proof_of_main_result} part \labelcref{item:main_result_introduction_i} of the following
\begin{theorem}[\cref{thm:main_result}, \cref{cor:realizing_isometries_as_mutations}]\label{thm:main_result_introduction}
	Let $X$ be a smooth projective surface over a field $k$ with $\chi(\mcO_X)=1$ and let $e_\bullet$ and $f_\bullet$ be exceptional bases of $\sfK_0^{\mathrm{num}}(X)$.
	\begin{enumerate}
		\item \label{item:main_result_introduction_i} There exists a $\bbZ$-linear automorphism $\phi\colon \sfK_0^{\mathrm{num}}(X) \to \sfK_0^{\mathrm{num}}(X)$ preserving the Euler pairing and the rank of elements such that $\phi(e_\bullet)$ can be transformed to $f_\bullet$ by a sequence of mutations and sign changes.
		\item \label{item:main_result_introduction_ii} If in addition $\rk \sfK_0^{\mathrm{num}}(X) \leq 12$, then $e_\bullet$ and $f_\bullet$ are related by a sequence of mutations and sign changes.
	\end{enumerate}
\end{theorem}
By definition, an exceptional basis of $\sfK_0^{\mathrm{num}}(X)$ is the class of a numerically exceptional collection of maximal length in $\sfK_0^{\mathrm{num}}(X)$, see \cref{def:basics_pseudolattices}. Thus, we can reformulate \cref{thm:main_result_introduction}~\labelcref{item:main_result_introduction_i} as:
Given two numerically exceptional collections $(E_1, \dots, E_n)$ and $(F_1, \dots, F_n)$ of maximal length on a surface $X$ with $\chi(\mcO_X)=1$ we can find a sequence of mutations and shifts $\sigma$ such that $\chi(\sigma(E_i), \sigma(E_j))=\chi(F_i, F_j)$ and $\rk\sigma(E_i)=\rk F_i$ holds for all $1\leq i, j \leq n$.

Allowing automorphisms of $\sfK_0^{\mathrm{num}}(X)$ preserving $\chi$ in addition to mutations and shifts was classically considered in the case of $X=\bbP^2$, where full exceptional collections can be interpreted as solutions of the Markov equation, see, e.g., \cite[\S\,7]{gorodentsev_kuleshov_helix_theory}. For lattices of higher rank this action was considered for instance in \cite{gorodentsev_helix_theory_and_nonsymmetrical_bilinear_forms}.

To prove \cref{thm:main_result_introduction}~\labelcref{item:main_result_introduction_i} we can restrict to the case of $X$ being either $\bbP^1 \times \bbP^1$ or a blow-up of $\bbP^2$ in a finite number of points by using Vial's classification recalled in \cref{thm:classification}.
Moreover, the group $\Aut (\sfK_0^{\mathrm{num}}(X))$ of isometries $\phi$ as in \cref{thm:main_result_introduction}~\labelcref{item:main_result_introduction_i} fits into a short exact sequence
\begin{equation*}
	1 \to (\Pic(X)/ \sim_{\mathrm{num}}) \to \Aut (\sfK_0^{\mathrm{num}}(X)) \to \Orth(\Pic(X)/\sim_{\mathrm{num}})_{K_X} \to 1,
\end{equation*}
where $\Orth(\Pic(X)/\sim_{\mathrm{num}})_{K_X}=\{f \in \Orth(\Pic(X)/\sim_{\mathrm{num}}) \mid f(K_X)=K_X \}$ is the stabilizer of the canonical class in the orthogonal group of $\Orth(\Pic(X)/\sim_{\mathrm{num}})$; see \cref{lem:orth_trans_lift_to_isometries}.

In \cref{se:9_points} we address the question how to lift \cref{thm:main_result_introduction}~\labelcref{item:main_result_introduction_i} to $\sfD^b(X)$ and prove \cref{thm:main_result_introduction}~\labelcref{item:main_result_introduction_ii}. The following two conditions are sufficient to deduce from \cref{thm:main_result_introduction}~\labelcref{item:main_result_introduction_i} that mutations and shifts act transitively on the set of full exceptional collections on $X$: 
\begin{enumerate}[label={(\alph*)}]
	\item \label{item:condition_i} The action of an isometry $\phi\colon \sfK_0^{\mathrm{num}}(X) \to \sfK_0^{\mathrm{num}}(X)$ as in \cref{thm:main_result_introduction}~\labelcref{item:main_result_introduction_i} can be realized as a sequence of mutations and shifts.
	\item \label{item:condition_ii} Two full exceptional collections sharing the same class in $\sfK_0^{\mathrm{num}}(X)$ can be transformed into each other by a sequence of mutations and shifts.
\end{enumerate}
If $X$ is a del Pezzo surface, the arguments of \cite{kuleshov_orlov_exceptional_sheaves_on_del_pezzo_surfaces} prove \labelcref{item:condition_ii}, see \cref{lem:exc_obj_del_pezzo}, and for the Hirzebruch surface $\Sigma_2$ the condition \labelcref{item:condition_ii} is verified in \cite[\S\,6]{ishi_okawa_uehara_exceptional_collections_on_sigma_2}.

The main theorem of Elagin--Lunts in \cite{elagin_lunts_on_full_exceptional_collections_of_line_bundles_on_del_pezzo_surfaces} states that any numerically exceptional collection consisting of line bundles on a del Pezzo surface is a full exceptional collection obtained from Orlov's blow-up formula applied to a minimal model.
We extend this result to the blow-up $X$ of $9$ points in very general position in $\bbP^2_\bbC$.
\begin{theorem}[\cref{cor:generalization_elagin_lunts}, \cref{thm:transitivity_on_9_points}]\label{thm:main_result_2}
	Let $X$ be the blow-up of $\bbP^2_\bbC$ in $9$ points in very general position. Then
	\begin{enumerate}
		\item \label{item:main_result_2_i} any numerically exceptional collection of maximal length consisting of line bundles is a full exceptional collection, and
		\item \label{item:main_result_2_ii} any two such collections are related by mutations and shifts.
	\end{enumerate}
\end{theorem}
The position of the $9$ points is discussed in \cref{rmk:pos_of_9_points}. Further, our results in \cite{krah_upcoming_work} show that the statements of \cref{thm:main_result_2} do not hold for blow-ups of more than $10$ points.

The proof of \cref{thm:main_result_2}~\labelcref{item:main_result_2_ii} is closely linked to the proof of \cref{thm:main_result_introduction}~\labelcref{item:main_result_introduction_ii}.
The key ingredient is the identification of the aforementioned group $\Orth(\Pic(X))_{K_X}$ with the Weyl group $W_X$ of a root system embedded in $\Pic(X)$, see \cref{lem:weyl_group_is_stabilizer}.
Although this lattice-theoretic equality holds for the blow-up of up to $9$ points regardless of their position, our argument relies on a result of Nagata \cite{nagata_on_rationl_surface_ii} which uses the actual geometry of $X$.
The equality $\Orth(\Pic(X))_{K_X}=W_X$ then enables us to verify condition \labelcref{item:condition_i} for the blow-up in up to $9$ points in very general position and thus we obtain  \cref{thm:main_result_2}~\labelcref{item:main_result_2_ii} and \cref{thm:main_result_introduction}~\labelcref{item:main_result_introduction_ii}.

In addition, our techniques provide a new proof of the fact that any two full exceptional collections on a del Pezzo surface are related by mutations and shifts; see \cref{cor:new_proof_trans_del_pezzo}. This result was proven in the first place by Kuleshov--Orlov in \cite[Thm.~7.7]{kuleshov_orlov_exceptional_sheaves_on_del_pezzo_surfaces}.

Finally \cref{sec:blow_up_10_points} discusses the lattice-theoretic behavior of the blow-up $X$ of $\bbP^2$ in $10$ points. In this case the Weyl group $W_X \subseteq \Orth(\Pic(X))_{K_X}$ has index two and $\Pic (X)$ admits an additional involution $\iota$ which fixes the canonical class $K_X$; see \cref{lem:weil_group_not_stabilizer_10_points}.
While the action of $W_X$ on exceptional collections of line bundles can be modeled by Cremona transformations of $\bbP^2$, the action of $\iota$ gives rise to an extraordinary numerically exceptional collection of line bundles.
In \cite{krah_upcoming_work} we show that the numerically exceptional collection obtained from $\iota$ is an exceptional collection of maximal length which is not full, provided the points are in general position.
As a consequence, $\sfD^b(X)$ contains a phantom subcategory and the braid group action on exceptional collections of maximal length is not transitive.
If one could verify condition \labelcref{item:condition_ii} for exceptional collections of maximal length on $X$, the results of \cite{krah_upcoming_work} would imply that the numerical bound in \cref{thm:main_result_introduction}~\labelcref{item:main_result_introduction_ii} is optimal, see \cref{rmk:sharpness_of_bound_on_rank}.

\begin{conventions}
	In this paper the term \emph{surface} always refers to a smooth projective variety of dimension $2$ over a field. The results in \cref{se:proof_of_main_result} are independent of the chosen base field, in \cref{se:9_points} we exclusively work over the complex numbers.
\end{conventions}

\begin{ack}
	This work is part of the author's dissertation, supervised by Charles Vial whom we wish to thank for numerous helpful discussions and explanations. Especially his encouragement to improve the results in this paper led to the developments in \cite{krah_upcoming_work}.
	Further, the author thanks Pieter Belmans for reading an earlier draft of this paper.
\end{ack}

\section{Numerically Exceptional Collections and Pseudolattices}
\label{se:basics}
We recall the necessary terminology of \emph{surface-like pseudolattices} as it is presented in \cite{kuznetsov_exceptional_collections_in_surface_like_categories}. Independently, the notion of a \emph{surface-type Serre lattice} was introduced in \cite{de_thanhoffer_de_volcsey_van_den_berg_on_an_analogue_of_the_markov_equation_for_exceptional_collections_of_length_4}.
After comparing both notions, we discuss the blow-up operation for pseudolattices in detail. Numerical blow-ups are explicitly mentioned in \cite{de_thanhoffer_de_volcsey_van_den_berg_on_an_analogue_of_the_markov_equation_for_exceptional_collections_of_length_4} but were already used in \cite{hille_perling_exceptional_sequences_of_invertible_sheaves_on_rational_surfaces, vial_exceptional_collections_and_the_neron_severi_lattice_for_surfaces, kuznetsov_exceptional_collections_in_surface_like_categories} in a slightly different manner.

\subsection{Exceptional collections}
Let $X$ be a smooth projective variety over a field $k$ and let $\sfD^b(X)\coloneqq \sfD^b(\mathrm{Coh}(X))$ be the bounded derived category of coherent sheaves on $X$. An object $E\in \sfD^b(X)$ is \emph{exceptional} if $\Hom(E_i, E_i)=k$ and $\Hom(E_i, E_i[l])=0$ for all $l\neq 0$. A \emph{full exceptional collection} in $\sfD^b(X)$ is a sequence of exceptional objects $(E_1, \dots, E_n)$ such that $E_1, \dots, E_n$ generate $\sfD^b(X)$ as a triangulated category and $\Hom(E_i, E_j[l])=0$ for all $l \in \bbZ$ whenever $i>j$.
When considering only their images in the Grothendieck group $\sfK_0(X)\coloneqq \sfK_0(\sfD^b(X))$ homomorphism spaces have to be exchanged with alternating sums over their dimensions. For this, let $$\chi(E, F)\coloneqq \sum_{j \in \bbZ} (-1)^j\dim_k \Hom(E,F[j])$$ be the \emph{Euler pairing}. It gives rise to a bilinear form on $\sfK_0(X)$ and an object $E\in \sfD^b(X)$ is called \emph{numerically exceptional} if $\chi(E,E)=1$.
\begin{definition}
	A \emph{numerically exceptional collection} in $\sfD^b(X)$ is a sequence of numerically exceptional objects $(E_1, \dots, E_n)$ such that $\chi(E_i, E_j)=0$ whenever $i>j$. The sequence is said to be of \emph{maximal length} if $[E_1], \dots, [E_n] \in \sfK_0(X)$ generate $\sfK_0^{\mathrm{num}}(X)$ as a $\bbZ$-module or equivalently if $n=\rk \sfK_0^{\mathrm{num}}(X)$.
\end{definition}
Here $\sfK_0^{\mathrm{num}}(X)\coloneqq \sfK_0(X)/\ker\chi$ denotes the \emph{numerical Grothendieck group}. Note that the left and right kernels of $\chi$ coincide thanks to Serre duality.
Clearly $\chi$ defines a non-degenerate bilinear form on $\sfK_0^{\mathrm{num}}(X)$. Therefore studying numerically exceptional collections can be reduced to studying non-degenerate $\bbZ$-valued bilinear forms, which will be formalized in the notion of a \emph{pseudolattice}.

\subsection{Surface-like pseudolattices}
We begin with recalling the notion of a pseudolattice in the sense of Kuznetsov.
\begin{definition}[{\cite[Def.~2.1]{kuznetsov_exceptional_collections_in_surface_like_categories}}]\label{def:basics_pseudolattices}
	A \emph{pseudolattice} is a finitely generated free abelian group $G$ together with a non-degenerate bilinear form $\chi \colon G \otimes_\bbZ G \to \bbZ$. An \emph{isometry} $\phi\colon (G, \chi_G) \to (H, \chi_H)$ between pseudolattices is a $\bbZ$-linear isomorphism which satisfies $\chi_G(v,w)=\chi_H(\phi(v), \phi(w))$ for all $v,w \in G$.
	\begin{itemize}
		\item The pseudolattice $(G,\chi)$ is \emph{unimodular} if $\chi$ induces an isomorphism $G \to \Hom_{\bbZ}(G, \bbZ)$.
		\item Let $e_\bullet = (e_1, \dots, e_n)$ be a basis of $G$, then $(\chi(e_i, e_j))_{i,j}$ is called the \emph{Gram matrix} with respect to $e_\bullet$.
		\item An element $e\in G$ is called \emph{exceptional} if $\chi(e,e)=1$.
		\item An ordered basis $e_\bullet$ is called \emph{exceptional basis} if the corresponding Gram matrix is upper unitriangular, i.e.~$\chi(e_i, e_j)=0$ whenever $i>j$ and $\chi(e_i, e_i)=1$ for all $i$.
		\item A \emph{Serre operator} is an isometry $S\colon G \to G$ satisfying $\chi(v, w)=\chi(w,S(v))$ for all $v, w \in G$.
	\end{itemize}
\end{definition}
Note that the lattice $G$ is unimodular if and only if the Gram matrix has determinant $\pm 1$. The Serre operator is unique, provided it exists, and if $G$ is unimodular, it is given by $M^{-1} M^T$, where $M$ is the Gram matrix of $\chi$ with respect to a chosen basis.
In case we need to pass to rational coefficients, we use the notation $G_\bbQ \coloneqq G\otimes_\bbZ \bbQ$ for a pseudolattice $G$ (or more generally for any abelian group).
\begin{definition}[{\cite[Def.~3.1]{kuznetsov_exceptional_collections_in_surface_like_categories}}]
	A pseudolattice $(G, \chi)$ is \emph{surface-like} if there exists a primitive element $\sfp \in G$ satisfying
	\begin{enumerate}
		\item $\chi(\sfp, \sfp)=0$,
		\item $\chi(\sfp, v)=\chi(v,\sfp)$ for all $v \in G$,
		\item $\chi$ is symmetric on $\sfp^\perp\coloneqq \{v\in G \mid \chi(\sfp,v)=0\}$.
	\end{enumerate}
	Such an element $\sfp$ is called a \emph{point-like element}.
\end{definition}
The terminology is justified by the following geometric example.
\begin{example}[Pseudolattices from surfaces]\label{ex:standard_example}
	Let $S$ be a smooth projective surface over a field $k$ which admits a $k$-valued point $i \colon \{x\} \hookrightarrow S$. For example if $S$ is rational, the existence of a $k$-valued point is guaranteed by the Lang--Nishimura Theorem. Let $G\coloneqq\sfK_0^{\mathrm{num}}(S)$ be the numerical Grothendieck group together with its Euler pairing. Then the class of the skyscraper sheaf $i_* k(x)=\mcO_{x}$ is a point-like element in $G$. An exceptional basis of $G$ is the same as the image of a numerically exceptional collection of maximal length on $S$ in $\sfK_0^{\mathrm{num}}(S)$.
\end{example}
\begin{remark}
	More generally, any $0$-cycle of degree $1$ in $\CH_0(S)$ provides a point-like element by realizing it as a Chern character of a complex of skyscraper sheaves. However all the surfaces we consider are rational, for that reason we will only consider point-like elements as in \cref{ex:standard_example}.
\end{remark}
From now on any surface-like pseudolattices $\sfK_0^{\mathrm{num}}(S)$, where $S$ is a surface over $k$, is implicitly assumed to be endowed with the Euler pairing and a point-like element given by the class of a skyscraper sheaf of a $k$-valued point. Recall that for a smooth projective surface $S$ the Chern character induces an isomorphism 
\begin{equation}\label{eq:chern_character_iso}
	\ch \colon \sfK_0^{\mathrm{num}}(S)_{\bbQ} \xrightarrow{\sim} \bbQ \oplus \left(\mathrm{Pic}(X)/\sim_{\mathrm{num}}\right)_{\bbQ} \oplus \bbQ,
\end{equation}
see, e.g., \cite[Lem.~2.1]{elagin_lunts_on_full_exceptional_collections_of_line_bundles_on_del_pezzo_surfaces}. For objects $E, F \in \sfD^b(X)$ with $e\coloneqq\rk E$ and $f\coloneqq\rk F$, Riemann--Roch yields
\begin{align}\label{eq:Riemann_Roch}
	\chi(E, F)={ }& ef\chi(\mcO_X) + \frac{1}{2} \left(f\cc_1(E)^2 + e\cc_1(F)^2 - 2\cc_1(E)\cc_1(F) \right) \\ &- \frac{1}{2}K_X (e\cc_1(F)-f\cc_1(E)) - (f\cc_2(E)+e\cc_2(F)).\nonumber
\end{align}
Given a surface-like pseudolattice $G$ with point-like element $\sfp$, we define the \emph{rank function} with respect to $\sfp$ to be $\sfr(-)\coloneqq\chi(\sfp,-)=\chi(-,\sfp)$. Then $\sfp^\perp = {^\perp \sfp} =\ker(\sfr)$ and we obtain the analogue of the decomposition in \cref{eq:chern_character_iso}.
\begin{lemma}[{\cite[Lem.~3.10, Lem.~3.11]{kuznetsov_exceptional_collections_in_surface_like_categories}}]
	If $G$ is a surface-like pseudolattice and $\sfp$ a point-like element, there is a complex
	\begin{equation*}
		\bbZ \xrightarrow{\sfp} G \xrightarrow{\sfr} \bbZ
	\end{equation*}
	with injective $\sfp$ and, if $G$ is unimodular, surjective $\sfr$. The middle cohomology of the above complex $\NS(G)\coloneqq\sfp^\perp / \sfp$ is a finitely generated free abelian group of rank $\rk(G)-2$.
	
	On $\NS(G)$ the pairing $-\chi$ induces a well-defined non-degenerate symmetric bilinear form $\sfq$, called the \emph{intersection form}, which also will be denoted by the usual product $-\cdot -$.
\end{lemma}
\begin{lemma}[{\cite[Lem.~3.12]{kuznetsov_exceptional_collections_in_surface_like_categories}}]\label{def:canonical_class_k}
	Let $G$ be a surface-like pseudolattice with point-like element $\sfp$ and let $\lambda \colon \bigwedge^2 G \to \sfp^\perp$ be the alternating map sending $v \wedge w \mapsto \sfr(v)w - \sfr(w)v$. Then there is a unique element $K_G \in \NS(G)_\bbQ$, called \emph{canonical class}, satisfying
	\begin{equation*}\label{eq:defining_eq_canonical_class}
		-\sfq(K_G, \lambda(v,w))=\chi(v,w)-\chi(w,v)
	\end{equation*}
	for all $v, w \in G$.
	If $G$ is unimodular, $K_G$ is integral, i.e.~$K_G \in \NS(G)$.
\end{lemma}

The pair $(\NS(G), \sfq)$ is called the \emph{Néron--Severi lattice} and $\NS(G)$ the \emph{Néron--Severi group}. One can check that  for a surface $S$ and pseudolattice $G=\sfK_0^{\mathrm{num}}(S)$ as in \cref{ex:standard_example} all these definitions agree with the usual ones up to sign. For example, via Riemann--Roch \cref{eq:Riemann_Roch} one computes $\chi(\mcO_x, \mcF)=-\rk \mcF$ for any coherent sheaf $\mcF$ on $S$ and $x\in S$ a $k$-valued point.

The following \cref{lem:orth_trans_lift_to_isometries}, which could not be found in the literature, will be important in the proof of \cref{thm:main_result_2}. For that reason, we provide a proof here.
\begin{lemma}[Self-isometries arise from orthogonal transformations]\label{lem:orth_trans_lift_to_isometries}
	Let $G$ be a surface-like pseudolattice of $\rk G\geq 3$ and let $\Aut(G)$ be the group of self-isometries $\phi \colon G \to G$ with $\phi(\sfp)=\sfp$.
	The map $\Psi \colon \Aut(G) \to \Orth(\NS(G))$ obtained by sending $\phi \in \Aut(G)$ to the induced orthogonal transformation of $\NS(G)$ defines a group homomorphism. If $G$ is unimodular, the image of $\Psi$ equals the stabilizer of the canonical class $\Orth(\NS(G))_{K_G}=\{f \in \Orth(\NS(G)) \mid f(K_G)=K_G \}$.
	Moreover if $G=\sfK_0^{\mathrm{num}}(X)$ for some surface $X$ with $\chi(\mcO_X)=1$ as in \cref{ex:standard_example}, the kernel of $\Psi$ can be identified with the subgroup of automorphisms given by twists with line bundles. In other words we obtain a short exact sequence
	\begin{equation*}
		1 \to (\Pic(X)/ \sim_{\mathrm{num}}) \to \Aut (\sfK_0^{\mathrm{num}}(X)) \to \Orth(\Pic(X)/\sim_{\mathrm{num}})_{K_X} \to 1.
	\end{equation*}
\end{lemma}
\begin{proof}
	Since $\sfr(-)=\chi(\sfp, -)$, any $\phi \colon G \to G$ which preserves the point-like element $\sfp$ preserves the rank of elements. Hence it induces an orthogonal transformation of $\NS(G)$ which fixes the canonical class $K_G$.
	If $G$ is unimodular, we can choose a rank $1$ vector $v_0\in G$ and a basis $\langle v_0, v_1, \dots, v_{n-1} \rangle$ such that $\sfp=v_{n-1}$ and $\sfp^\perp = \langle v_1, \dots, v_{n-1}\rangle$. By adding suitable multiples of $\sfp$ to the $v_i$ we can arrange $\chi(v_i, v_0)=0$ for all $1 \leq i \leq n-2$. Now $\chi(v_0, v_i)=-\sfq(K_G,v_i)$ for $1 \leq i \leq n-2$. Thus any $\bar{\phi} \in \Orth(\NS(G))_{K_G}$ can be lifted to an isometry $\phi$ of $G$ preserving $\sfp$ and fixing $v_0$ by choosing $\phi(v_i)$ to be the unique lift of $\bar{\phi}(v_i) \in \sfp^\perp / \sfp$ which satisfies $\chi(\phi(v_i), v_0)=0$.
	The construction of the lift depends on the choice of $v_0$ and for any other choice $v_0'$ with $\chi(v_0, v_0)=\chi(v_0', v_0')$ and $\sfr (v_0')=1$ there exists exactly one lift $\phi$ which maps $v_0$ to $\phi(v_0)=v_0'$.
	
	Assume $G=\sfK_0^{\mathrm{num}}(X)$ for some surface $X$ with $\chi(\mcO_X)=1$ and choose the initial $v_0$ to be the class $-[\mcO_X]$.
	For any numerically exceptional object $E$ of rank $1$, Riemann--Roch \cref{eq:Riemann_Roch} implies that $\cc_2(E)=0$.
	Since 
	$$\ch(E)=\left(\rk E, \cc_1(E), \frac{1}{2}(\cc_1(E)^2-2\cc_2(E))\right),$$
	the condition $\cc_2(E)=0$ implies $[E]=-[\mcO_X(\cc_1(E))]$ in $\sfK_0^{\mathrm{num}}(X)$.
	Now twisting with $\mcO_X(\cc_1(E))$ defines an isometry of $G$ which maps $v_0=-[\mcO_X]$ to $[E]$. Let $v_i=[F]$ for an object $F \in \sfD^b(X)$. Then $\ch(F)=(0, \cc_1(F), d)$, $d \in \bbQ$. Multiplicativity of the Chern character gives $$\ch(F(c_1(E)))=(0, \cc_1(F), d)\cdot (1, \cc_1(E), d')=(0, \cc_1(F), \cc_1(F)\cc_1(E)+d),$$ for some $d' \in \bbQ$. We observe that the first Chern class of $F$ is invariant under twisting with a line bundle and also twisting with a line bundle does not change the point-like element defined by a skyscraper sheaf. Thus, twisting with $\mcO_X(\cc_1(E))$ is the unique automorphism of $G$ which maps $v_0$ to $v_0'=[E]$ and induces the identity on $\NS(G)$.
\end{proof}
\begin{definition}[{\cite[Def.~4.1, Lem.~4.2, Def.~4.3]{kuznetsov_exceptional_collections_in_surface_like_categories}}]
	A surface-like pseudolattice $G$ is called \emph{geometric} if $(\NS(G), \sfq)$ has signature $(1, \rk G -3)$, the canonical class $K_G$ is integral and $K_G$ is \emph{characteristic}, i.e.~$\sfq(D,D) \equiv \sfq(K_G,D) (\mod 2)$ for all $D \in \NS(G)$.	
	A surface-like pseudolattice $G$ is \emph{minimal} if it has no exceptional elements of rank zero. Equivalently $\NS(G)$ does not contain any $(-1)$-class.
\end{definition}
It turns out that such geometric pseudolattices can be classified if we restrict to defect zero pseudolattices. Here the \emph{defect} of $G$ is the integer
\begin{equation*}
	\delta (G)\coloneqq K_G^2 + \rk(G) - 12.
\end{equation*}
If $G$ is obtained as in \cref{ex:standard_example} from a surface $S$ which has an exceptional structure sheaf $\mcO_S$, one can compute $\delta(G)=0$; see \cite[Lem.~5.4]{kuznetsov_exceptional_collections_in_surface_like_categories}. In general the defect can be interpreted as a suitable numerical replacement of the condition $\chi(\mcO_S)=1$.
\begin{theorem}[{\cite[Thm.~3.1]{vial_exceptional_collections_and_the_neron_severi_lattice_for_surfaces}, \cite[Thm.~5.11]{kuznetsov_exceptional_collections_in_surface_like_categories}}]\label{thm:classification}
	Let $G$ be a unimodular geometric pseudolattice of rank $n \geq 3$
	and zero defect such that $G$ represents $1$ by a rank $1$ vector, i.e.~there exists $v\in G$ of rank $1$ such that $\chi(v,v)=1$. Then the following holds:
	\begin{itemize}
		\item $n=3$ and $K_G = -3H$ for some $H \in \NS(G)$ if and only if $G$ is isometric to $\sfK_0^{\mathrm{num}}(\bbP^2)$;
		\item $n=4$, $\NS(G)$ is even and $K_G=-2H$ for some $H \in \NS(G)$ if and only if $G$ is isometric to $\sfK_0^{\mathrm{num}} (\bbP^1 \times \bbP^1)$;
		\item $n\geq 4$, $\NS(G)$ is odd and $K_G$ is primitive if and only if $G$ is isometric to $\sfK_0^{\mathrm{num}}(X_{n-3})$.
	\end{itemize}
	Here $X_{n-3}$ is the blow-up of $\bbP^2$ in $n-3$ points.
	Furthermore, $G$ has an exceptional basis if and only if one of the three possibilities listed above is satisfied.
\end{theorem}
\begin{remark}\label{rmk:defect_zero_for_lattices_from_surfaces}
	In fact, by \cite[Thm.~3.1]{vial_exceptional_collections_and_the_neron_severi_lattice_for_surfaces} the condition $\delta(G)=0$ is a necessary condition for admitting an exceptional basis if the pseudolattice results from taking the numerical Grothendieck group of a smooth projective surface with $\chi(\mcO_S)=1$.
\end{remark}

Let $G$ be a surface-like pseudolattice with Serre operator $S$. Let $v \in G$, then 
\begin{equation*}
	\chi(\sfp, (S-1)(v)) = \chi(v, \sfp) - \chi(\sfp, v)=0.
\end{equation*}
Furthermore, \cite[Lem.~3.14]{kuznetsov_exceptional_collections_in_surface_like_categories} shows that $S-1$ maps $\sfp^\perp$ to $\bbZ \sfp$. Thus, we obtain a decreasing filtration
\begin{equation*}
	F^3G=0 \subseteq F^2 G= \bbZ \sfp \subseteq F^1 G = \sfp^\perp \subseteq F^0 G = G
\end{equation*}
such that $S-1$ maps $F^i G$ to $F^{i+1} G$. If $G$ is unimodular, the rank map induces an isomorphism $\sfr \colon G / \sfp^\perp \to \bbZ$, thus the above filtration defines a so-called \emph{codimension filtration}.
\begin{definition}[{\cite[Def.~5.1.1]{de_thanhoffer_de_volcsey_van_den_berg_on_an_analogue_of_the_markov_equation_for_exceptional_collections_of_length_4}}]
	Let $G$ be a pseudolattice with Serre operator $S$ and let $V\coloneqq G_\bbQ$. A \emph{codimension filtration} on $V$ is a filtration
	\begin{equation*}
		0=F^3V \subseteq F^2 V \subseteq F^1 V \subseteq F^0 V = V
	\end{equation*}
	such that $(S-1)(F^iV) \subseteq F^{i+1}V$, $\dim F^0 V / F^1 V = \dim F^2 = 1$ and $\chi(F^1 V, F^2V)=0$.
\end{definition}
Conversely, any codimension filtration gives rise to a point-like element by choosing a generator of $F^2 G=F^2V \cap G$. This yields a 1:1-correspondence
\begin{equation*}
	\{\text{codimension filtrations} \; F^\bullet \; \text{on} \; G\} \leftrightarrow \{\text{point-like elements} \; \sfp\} /\{\pm 1\}.
\end{equation*}
We will refer to both of them, a point-like element and a codimension filtration, as a \emph{surface-like structure} on the pseudolattice $G$. In \cref{ex:standard_example} the codimension filtration coincides with the topological codimension filtration, as discussed in \cite[Ex.~3.5]{kuznetsov_exceptional_collections_in_surface_like_categories}.

\subsection{Mutations}\label{sec:mutations}
Given $e \in G$ we define the \emph{left mutation} $\sfL_e$ and its \emph{right mutation} $\sfR_e$ as
\begin{align*}
	\sfL_e(v)&\coloneqq v-\chi(e,v)e\\
	\sfR_e(v)&\coloneqq v-\chi(v,e)e
\end{align*}
for all $v\in G$.
Given an exceptional basis $e_\bullet =(e_1, \cdots,e_n)$ of $G$ we define
\begin{align*}
	\sfL_{i,i+1}(e_\bullet) &\coloneqq (e_1, \cdots, e_{i-1}, \sfL_{e_i}(e_{i+1}), e_i, e_{i+2}, \cdots, e_n), \\
	\sfR_{i,i+1}(e_\bullet) &\coloneqq (e_1, \cdots, e_{i-1}, e_{i+1}, \sfR_{e_{i+1}}( e_i), e_{i+2}, \cdots, e_n).
\end{align*}
The sequences are again exceptional bases and the above operations are mutually inverse.
By construction, these mutations match the known mutations of exceptional collections if $G=\sfK_0^{\mathrm{num}}(S)$ as in \cref{ex:standard_example}. Indeed, if $S$ is a surface and $E \in \sfD^b(S)$ an exceptional object, the \emph{left mutation} $\sfL_E$ and \emph{right mutation} $\sfR_E$ are defined as
\begin{align*}
	\sfL_E(F)&\coloneqq \mathrm{Cone}\left(E \otimes \RHom(E,F)\xrightarrow{\mathrm{ev}} F\right) \; \text{and}  \\ \sfR_E(F)&\coloneqq \mathrm{Cone}\left(F \xrightarrow{\mathrm{ev}^\vee} E \otimes \RHom(F,E)^\vee\right)[-1]
\end{align*}
for any object $F \in \sfD^b(S)$. Note that by construction the diagram
\[
\begin{tikzcd}
	\sfD^b(S) \rar \dar["\sfM_{E}"]  & \sfK_0^{\mathrm{num}}(S)\dar["\sfM_{[E]}"] \\
	\sfD^b(S) \rar & \sfK_0^{\mathrm{num}}(S) 
\end{tikzcd}
\]
commutes, where $\sfM_{E}=\sfL_E$ or $\sfM_{E}=\sfR_E$.

Moreover, if $\sfD^b(S)=\langle E_1, \dots, E_n \rangle=\langle E_\bullet \rangle$ is a full exceptional collection, the sequences
\begin{align*}
	\sfL_{i,i+1}(E_\bullet) &\coloneqq (E_1, \cdots, E_{i-1}, \sfL_{E_i}(E_{i+1}), E_i, E_{i+2}, \cdots, E_n), \\
	\sfR_{i,i+1}(E_\bullet) &\coloneqq (E_1, \cdots, E_{i-1}, E_{i+1}, \sfR_{E_{i+1}}( E_i), E_{i+2}, \cdots, E_n).
\end{align*}
are again full exceptional collections.
Already on the level of $\sfD^b(S)$ the operations $\sfL_{i,i+1}$ and $\sfR_{i,i+1}$ give rise to an action of the braid group $\mathfrak{B}_n$, see, e.g., \cite[Prop.~2.1]{bondal_polishchuk_homological_properties_of_associative_algebras_the_method_of_helices}. Together with $\bbZ^n$ acting by shifts, this yields an action of the semidirect product $\bbZ^n \rtimes \mathfrak{B}_n$ on the set of full exceptional collections, where the homomorphism $\mathfrak{B}_n \to \Aut(\bbZ^n)$ is the composition of the canonical map $\mathfrak{B}_n \to \mathfrak{S}_n$ and the action of $\mathfrak{S}_n$ on $\bbZ^n$ by permutations.
If two exceptional bases lie in the same orbit of the $\bbZ^n \rtimes \mathfrak{B}_n$-action, we say the exceptional collections are \emph{related by mutations up to shifts}.

On the level of $\sfK_0^{\mathrm{num}}(S)$ shifts result in sign changes. More generally, if $G$ is a pseudolattice of rank $n$ with exceptional basis, then $\{\pm 1\}^n \rtimes \mathfrak{B}_n$ acts on the set of exceptional bases, where $\{\pm 1\}^n$ acts by changing signs of basis elements. Moreover, this action commutes with the action of isometries $\phi \colon G \to G$.
If two exceptional bases lie in the same orbit of $\{\pm 1\}^n \rtimes \mathfrak{B}_n$, we say the exceptional bases are \emph{related by mutations up to signs}.
In this paper, we will only consider pseudolattices with surface-like structure. If we write that two exceptional bases $e_\bullet, f_\bullet$ are \emph{related by mutations up to signs and isometry} we always mean that there exists an isometry $\phi \colon G \to G$ which preserves the point-like element $\phi(\sfp)=\sfp$ and $\phi(e_\bullet)$ and $f_\bullet$ are related by mutations up to signs.

Let $G$ be a surface-like pseudolattice with exceptional basis. We will frequently mutate to norm-minimal bases, where the \emph{norm} of an exceptional basis $e_\bullet=(e_1, \dots, e_n)$ is the number $\sum_i \sfr(e_i)^2$. We say an exceptional basis is \emph{norm-minimal} if there is no exceptional basis related by mutations and sign changes with smaller norm.
Recall that due to the work of Perling, norm-minimal exceptional bases can be understood via \emph{Perling's algorithm}:
\begin{theorem}[{\cite[Thm.~5.8]{kuznetsov_exceptional_collections_in_surface_like_categories}}, cf.~{\cite[Cor.~9.12, Cor.~10.7]{perling_combinatorial_aspects_of_exceptional_sequences_on_rationl_surfaces}}]\label{thm:perlings_algorithm}
	Let $G$ be a geometric surface-like pseudolattice. Any exceptional basis in $G$ can be transformed by mutations and sign changes into a norm-minimal exceptional basis consisting of $3$ or $4$ elements of rank $1$ and all other elements of rank $0$.
\end{theorem}

\subsection{Blow-up and blow-down}\label{se:blow_up_down}
We recall the classical blow-up and blow-down construction for surface-like pseudolattices and give a detailed discussion of \cite[\S\,5]{de_thanhoffer_de_volcsey_van_den_berg_on_an_analogue_of_the_markov_equation_for_exceptional_collections_of_length_4} as we make use of these observations in \cref{se:proof_of_main_result}.
Let $G$ be a unimodular surface-like pseudolattice with point-like element $\sfp$. We denote the induced codimension filtration by $F^\bullet G$. Let $e_\bullet=(e_1, \dots, e_n)$ be a basis of $G$ and let $M$ be the Gram matrix of the pairing $\chi$ with respect to this basis. Choosing an element $z\in F^2G = \bbZ \sfp$, we construct the \emph{numerical blow-up} of $G$ at $z$ as follows: We extend the lattice $G$ by adding a formal element $f$, i.e.~we consider the free abelian group $\Bl_z G\coloneqq \bbZ f \oplus G$. The pairing $\chi_{\mathrm{new}}$ on $\Bl_z G$ is defined via
\begin{align*}
	\chi_{\mathrm{new}} \vert_{G\otimes G} &\coloneqq  \chi, \\
	\chi_{\mathrm{new}} (g,f) &\coloneqq 0 \; \text{for all} \; g \in G,\\
	\chi_{\mathrm{new}} (f,f) &\coloneqq 1,  \\
	\text{and } \chi_{\mathrm{new}} (f,g) &\coloneqq \chi(z,g) \; \text{for all} \; g \in G.
\end{align*}
In abuse of notation we write $\chi$ also for the pairing on $\Bl_z G$.
As outlined below, this definition matches the geometric situation of a blow-up.
The Gram matrix with respect to the basis $(f, e_1, \dots, e_n)$ is of the form
\begin{equation*}\label{eq:formula_gram_matrix_blow_up}
\left(\begin{array}{@{}c|c@{}}
	1
	& \begin{matrix}
		\chi (z, e_1) & \cdots & \chi (z, e_n)
	\end{matrix} \\
	\hline
	\begin{matrix}
		0 \\
		\vdots \\
		0
	\end{matrix} & M
\end{array}\right).
\end{equation*}
Note that $\Bl_z G$ is again unimodular and surface-like with point-like element $\sfp \in G \subseteq \bbZ f \oplus G$. The latter follows from writing $z=n\sfp$ for some $n \in \bbZ$ which shows $\chi(\sfp, f)=0=\chi(z, \sfp)=\chi(f,\sfp)$.
The orthogonal complement of $\sfp$ in $\Bl_z G$ is $F^1 \Bl_z G = F^1 G \oplus \bbZ f$ and $\chi$ is symmetric on $F^1 G \oplus \bbZ f$ as it is symmetric on both summands and $\chi(F^1 G, f)= 0= n \chi( \sfp, F^1G)=\chi(f, F^1 G)$. In particular, $F^2 \Bl_z G = \bbZ \sfp = F^2 G$. Therefore the point-like element $\sfp$ does not change under blow-up; this allows us to blow up the same element multiple times.
Note that the image of $f$ in 
\begin{equation*}
	\NS(\Bl_z G)=\NS(G)\overset{\perp}{\oplus} \bbZ f
\end{equation*}
defines an element of self-intersection $-1$. It is the analogue of a $(-1)$-curve and can be blown down, but in contrast to the geometric setting, we cannot detect whether a divisor of self-intersection $-1$ is an actual curve or not.

Again we compare the construction to the geometric one (cf.~\Cref{ex:standard_example}).
Let $S$ be a smooth projective surface and let $\tilde{S}$ be the blow-up at a point $p\in S$ with exceptional divisor $E$:
\[
\begin{tikzcd}
	\tilde{S} \rar["\pi"] & S \\
	E \rar["\psi"] \uar[hook, "j"] & \{p\}.\uar[hook, "i"]
\end{tikzcd}
\]
For $\mcF \in \sfD^b(S)$ a Riemann--Roch computation shows $$\chi_{\tilde{S}}(j_* \mcO_E(-1), \pi^* \mcF)=-\rk(\mcF)=\chi_S(\mcO_p, \mcF),$$ see \cite[Ex.~4.1]{perling_combinatorial_aspects_of_exceptional_sequences_on_rationl_surfaces}.
Finally, Orlov's blow-up formula yields a semi-orthogonal decomposition
\begin{equation*}
	\sfD^b(\tilde{S})= \langle j_*(\psi^* \sfD^b(\{p\}) \otimes \mcO_E(-1) ), \mathrm{L} \pi^* \sfD^b(S) \rangle = \langle j_*\mcO_E(-1), \mathrm{L} \pi^* \sfD^b(S) \rangle
\end{equation*}
which coincides with the numerical blow-up construction.

The inverse operation on a unimodular surface-like pseudolattice $G$ is the \emph{blow-down} or \emph{contraction}. Let $f \in G$ be a rank zero vector such that $\sfq(f,f)=-\chi(f,f)=-1$. Then the contraction of $f$ is the lattice $G_f\coloneqq  {^\perp f} =\{v\in G \mid \chi(v,f)=0\}\subseteq G$ with pairing $\chi \vert_{^\perp f \otimes {^\perp f}}$. The pseudolattice $G_f$ is again surface-like with point-like element $\sfp$ and unimodular; see \cite[Lem.~5.1]{kuznetsov_exceptional_collections_in_surface_like_categories}. If $G$ is geometric, so is $G_f$.
In the following we prove a slightly modified version of \cite[Lem.~5.1]{de_thanhoffer_de_volcsey_van_den_berg_on_an_analogue_of_the_markov_equation_for_exceptional_collections_of_length_4}, which will be a key tool towards establishing \cref{thm:main_result}.
\begin{proposition}
	Let $G$ be a unimodular surface-like pseudolattice and $f \in G$ a rank zero vector of self-intersection $-1$. Denote by $S$ the Serre operator of $G$, then $z\coloneqq (S-1)(f) \in F^2 G_f$ defines an element such that $\Bl_z G_f = G$.
\end{proposition} 
\begin{proof}
	Since $f\in F^1 G = \sfp ^\perp$, we know that $(S-1)(f) \in F^2G = \bbZ \sfp$ is a multiple of $\sfp$ and lies in $G_f$. Let $H\coloneqq  \Bl_z (G_f)=\bbZ g \oplus G_f$ be the blow-up of $G_f$ at $z$. Then the pairing on $H$ extends the pairing of $G_f$ with the property that $g$ is an element of rank zero and of self-intersection $-1$. Consider the morphism $G \to H$ sending $f \mapsto g$ and $v\mapsto v$ for all $v\in G_f$. We verify that this is an isometry:
 	Obviously $\chi(g,g)=1=\chi(f,f)$ and $\chi(v, f)=0=\chi(v,g)$ for $v\in G_f$. Let $v \in G_f$, then
 	\begin{equation*}
 		\chi(g,v)=\chi(z,v)=\chi(v,z)=\chi(v,S(f))-\chi(v,f)=\chi(f,v),
 	\end{equation*}
 	where we have used that $z$ is a multiple of $\sfp$ and $\chi(v,f)=0$ for all $v\in G_f$.
\end{proof}
Clearly $(\Bl_z G)_f =G$, thus blow-up and blow-down are mutually inverse.

\begin{remark}
	Comparing the blow-down construction described above to the construction in \cite[\S\,5]{kuznetsov_exceptional_collections_in_surface_like_categories}, one observes that the contraction of an exceptional element of rank zero can also be defined as the right orthogonal $f^\perp$.
\end{remark}

We end this section by recalling formulae for the defect of the contraction.
\begin{lemma}[{\cite[Lem.~5.7]{kuznetsov_exceptional_collections_in_surface_like_categories}}]\label{lem:formula_defect_contraction}
	Let $G$ be a surface-like pseudolattice and $e\in G$ an exceptional element of rank zero. Then the defect of $G$ equals
	\begin{equation*}
		\delta (G) = \delta (G_e) + (1-\sfq(K_G, e)^2).
	\end{equation*}
	If $G$ is geometric, then $\delta (G) \leq \delta (G_e)$ with equality if and only if $\sfq(K_G, e)=\pm 1$.
\end{lemma}
In the same manner a formula for the degree of the blow-up was obtained in \cite[Lem.~5.2.1]{de_thanhoffer_de_volcsey_van_den_berg_on_an_analogue_of_the_markov_equation_for_exceptional_collections_of_length_4}. The \emph{degree} of a unimodular surface-like pseudolattice $G$ is $\deg (G)=K_G^2$ and is related to the defect by the formula
\begin{equation}\label{eq:degree_in_term_of_rank}
	\deg (G) = 12 + \delta (G) - \rk (G).
\end{equation}

\begin{lemma}[{\cite[Lem.~5.2.1]{de_thanhoffer_de_volcsey_van_den_berg_on_an_analogue_of_the_markov_equation_for_exceptional_collections_of_length_4}}]\label{lem:formula_degree_blow_up}
	Let $G$ be a unimodular surface-like pseudolattice and let $\sigma \in G$ be an element such that its image $\bar{\sigma}$ generates $\Bl_z G / F^2 \Bl_z G\cong \bbZ$. Then $\deg \Bl_z G = \deg G - \chi(\sigma, z)^2$.
\end{lemma}
Above \cref{lem:formula_degree_blow_up} requires a justification in our context, as it is possibly not clear that the canonical class in the sense of \cite{de_thanhoffer_de_volcsey_van_den_berg_on_an_analogue_of_the_markov_equation_for_exceptional_collections_of_length_4} coincides with the one in \cref{def:canonical_class_k}.
\begin{proof}[Proof of \cref{lem:formula_degree_blow_up}]
	The image $\omega$ of $(S-1)(\sigma)$ in $\NS(G)$ is the canonical class of $G$ in the sense of \cite[Def.~3.5.1]{de_thanhoffer_de_volcsey_van_den_berg_on_an_analogue_of_the_markov_equation_for_exceptional_collections_of_length_4} and in \cite[Lem.~5.2.1]{de_thanhoffer_de_volcsey_van_den_berg_on_an_analogue_of_the_markov_equation_for_exceptional_collections_of_length_4} the statement is shown for $\deg G\coloneqq  \sfq(\omega, \omega)$. Therefore it is enough to show:
	\begin{claim}
		Let $G$ be a unimodular surface-like pseudolattice, $\sigma \in G/ F^2G$ a generator and $\omega=(S-1)(\sigma) \in \NS(G)$. Then $\omega$ satisfies the defining equation in \Cref{def:canonical_class_k} up to sign, i.e.
		\begin{equation*}
			\pm \sfq (\omega, \lambda(v,w))= \chi(v,w)-\chi(w,v)
		\end{equation*}
		for all $v,w \in G$ and $\lambda$ as in \Cref{def:canonical_class_k}.
	\end{claim}
	\paragraph{\emph{Proof of the Claim}}
		Since $G$ is unimodular, the rank map induces an isomorphism $\sfr \colon G/F^2 G \to \bbZ$. Let $\sigma \in G$ be a vector such that $\bar{\sigma}$ generates $G/F^2 G$. Up to possibly replacing $\sigma$ by $-\sigma$ we can write any $v\in G$ as $\sfr (v)\sigma + \tau(\sigma)$ with $\tau(\sigma) \in F^2G$.
		Let $\omega = (S-1)(\sigma)$ be the canonical class defined by $\sigma$ and let $d(v)\coloneqq \sfq(\tau(v), \omega)$ for all $v\in G$.
		By \cite[Prop.~3.6.2]{de_thanhoffer_de_volcsey_van_den_berg_on_an_analogue_of_the_markov_equation_for_exceptional_collections_of_length_4} the equality
		\begin{equation}\label{eq:symmetric_with_determinant}
			\chi(v, w)-\chi(w,v) = \det \begin{pmatrix}
				d(v) & d(w) \\
				\sfr (v) & \sfr (w)
			\end{pmatrix}
			= \sfr (w) \sfq(\tau(v), \omega) - \sfr(v) \sfq (\tau(w), \omega)
		\end{equation}
		holds for all $v, w\in G$.
		Let $\lambda$ be the alternating form as in \Cref{def:canonical_class_k}. Then
		\begin{align*}
			-\sfq(\omega, \lambda(v,w)) 
			&=  -\sfq(\sfr (v)( \sfr (w) \sigma + \tau(w)   )   -    \sfr (w)( \sfr (v) \sigma + \tau(v)   ) , \omega) \\
			&=  -\sfq(    \sfr(v) \sfr(w)   \sigma + \sfr (v) \tau(w)- \sfr(w) \sfr(v) \sigma - \sfr (w) \tau(v)         , \omega) \\
			&=   \sfq( \sfr (w) \tau(v)   - \sfr (v) \tau(w), \omega )
		\end{align*}
		combined with \Cref{eq:symmetric_with_determinant} proves the claim.
\end{proof}

\section{Proof of \texorpdfstring{\Cref{thm:main_result_introduction}~\labelcref{item:main_result_introduction_i}}{Theorem \ref{thm:main_result_introduction}~\ref{item:main_result_introduction_i}}}
\label{se:proof_of_main_result}
Throughout this section let $X_k$ be the blow-up of $\bbP^2$ in $k$ distinct points and let $G_k\coloneqq \sfK_0^{\mathrm{num}}(X_k)$ be the pseudolattice obtained from $X_k$. Using Vial's classification, see \cref{thm:classification} and \cref{rmk:defect_zero_for_lattices_from_surfaces}, we can rephrase \cref{thm:main_result_introduction}~\labelcref{item:main_result_introduction_i} as follows:
\begin{theorem}\label{thm:main_result}
	Let $e_\bullet$ and $f_\bullet$ be two exceptional bases of $G_k$ or of $\sfK_0^{\mathrm{num}}(\bbP^1\times\bbP^1)$. Then there exists an isometry $\phi \colon G \to G$ preserving the surface-like structure, i.e.~$\phi(\sfp)= \sfp$, such that $\phi(e_\bullet)$ and $f_\bullet$ are related by mutations up to signs.
\end{theorem}
In preparation for the proof of \Cref{thm:main_result} we compute an explicit form of the pseudolattices $G_k$.
The surface $\bbP^2$ admits a full exceptional sequence consisting of line bundles, namely the Beilinson sequence $\sfD^b(\bbP^2)=\langle \mcO_{\bbP^2}, \mcO_{\bbP^2}(1), \mcO_{\bbP^2}(2) \rangle$. This yields an exceptional basis of the numerical Grothendieck group $G_0\coloneqq \sfK_0^{\mathrm{num}}(\bbP^2)=\bbZ [\mcO_{\bbP^2}] \oplus \bbZ [\mcO_{\bbP^2}(1)] \oplus \bbZ [\mcO_{\bbP^2}(2)]$ with Gram matrix
\begin{equation*}
	M_0\coloneqq  \begin{pmatrix}
		1 & 3 & 6 \\
		0 & 1 & 3 \\
		0 & 0 & 1
	\end{pmatrix}.
\end{equation*}

\begin{lemma}
	Let $e_\bullet=(e_1, e_2, e_3)$ be an exceptional basis of $G_0$ with Gram matrix $M_0$, then a point-like element is given by $\sfp\coloneqq e_3 -2e_2 + e_1$.
\end{lemma}
\begin{proof}
	For a closed point $i \colon \{x\} \hookrightarrow {\bbP^2}$ the skyscraper-sheaf $i_*k(x)=\mcO_x$ admits a Koszul resolution
	\begin{equation*}
		[0 \to \mcO_{\bbP^2}(-2)=\bigwedge^2 \mcO_{\bbP^2}(-1)^{\oplus 2} \to \mcO_{\bbP^2}(-1)^{\oplus 2} \to \mcO_{\bbP^2} \to 0] \cong \mcO_x.
	\end{equation*}
	Thus, we obtain after twisting by $\mcO_{\bbP^2}(2)$
	\begin{equation*}
		[\mcO_x] = [\mcO_{\bbP^2}(2)]-2[\mcO_{\bbP^2}(1)] + [\mcO_{\bbP^2}] = e_3 -2e_2 + e_1 \in \sfK_0^{\mathrm{num}}(\bbP^2).\qedhere
	\end{equation*}
\end{proof}
\begin{remark}
	The point-like element can also be computed directly from the pseudolattice, using the explicit description of \cite[Lem.~3.3.2]{de_thanhoffer_de_volcsey_van_den_berg_on_an_analogue_of_the_markov_equation_for_exceptional_collections_of_length_4}. Namely if $V\coloneqq  G_\bbQ$, then $F^2 V= \Im(S-1)^2$ and $F^2 V = \bbQ \sfp$. Thus $\sfp$ spans the line $\Im(S-1)^2$ over $\bbQ$ and is primitive.
	In the case of $\bbP^2$ one computes $S= M_0^{-1}M_0^T$ and
	\begin{equation*}
		(S-1)^2 = \begin{pmatrix}
			9 & 9 & 9 \\
			-18 & -18 & -18 \\
			9 & 9 & 9
		\end{pmatrix}.
	\end{equation*}
	Hence, $\sfp=\pm (1, -2,1)$.
\end{remark}

By the blow-up formula we compute Gram matrices $M_k$ of the pseudolattices $G_k$, namely
\begin{equation*}
	M_k\coloneqq  \begin{pmatrix}
		1 & 0 & \cdots & 0 & 1 & 1 & 1 \\
		0 & 1 & \cdots & 0 & 1 & 1 & 1 \\
		\vdots &   &\ddots & \vdots & \vdots & &\vdots \\
		0 & 0 & \cdots & 1 & 1 & 1 & 1 \\
		0 & 0 & \cdots & 0 & 1 & 3 & 6 \\
		0 & 0 & \cdots & 0 & 0 & 1 & 3 \\
		0 & 0 & \cdots & 0 & 0 & 0 & 1 
	\end{pmatrix}=
\left(\begin{array}{@{}c|c@{}}
	\mathrm{id}_{k\times k}
	& \begin{matrix}
		1 & 1 & 1 \\
		\vdots & \vdots &\vdots \\
		1 & 1 & 1
	\end{matrix} \\
	\hline
	0 & M_0
\end{array}\right),
\end{equation*}
since $\chi(-\sfp, e_i)=1$ for $i=1, 2, 3$.
Denote by $b_1, \dots, b_k, e_1, e_2, e_3$ the exceptional basis corresponding to this Gram matrix. The elements $b_i$ are all orthogonal to $\sfp$, so of rank zero and the corresponding images in $\NS(G)$ have self-intersection $-1$. Here, we have numerically blown up the point $-\sfp$ in order to obtain only positive signs in the Gram matrix.
We first verify \cref{thm:main_result} in the minimal cases:
\begin{proposition}[{\cite[Cor.~4.25]{kuznetsov_exceptional_collections_in_surface_like_categories}}]\label{prop:result_for_P2}
	Any two exceptional bases in $\sfK_0^{\mathrm{num}}(\bbP^2)$ are related by mutations up to sign and isometry.
\end{proposition}
\begin{proof}
	By \cite[Cor.~4.25]{kuznetsov_exceptional_collections_in_surface_like_categories} norm-minimal exceptional bases of $\sfK_0^{\mathrm{num}}(\bbP^2)$ correspond to the Beilinson sequence $\langle \mcO_{\bbP^2}, \mcO_{\bbP^2}(1), \mcO_{\bbP^2}(2) \rangle$. Therefore any two exceptional bases are related by mutations up to sign and isometry.
\end{proof}
\begin{proposition}\label{prop:result_for_P1_x_P1}
	Any two exceptional bases in $\sfK_0^{\mathrm{num}}(\bbP^1 \times \bbP^1)$ are related by mutations up to sign and isometry.
\end{proposition}
\begin{proof}
	If $G$ admits a norm-minimal basis consisting of objects of nonzero rank, $G$ is isometric to $\sfK_0^{\mathrm{num}} (\bbP^1 \times \bbP^1)$ and the norm-minimal basis corresponds to one of the full exceptional collections
	\begin{equation*}
		\sfD^b(\bbP^1 \times \bbP^1)= \langle \mcO_{\bbP^1\times\bbP^1}, \mcO_{\bbP^1\times\bbP^1} (1,0), \mcO_{\bbP^1\times\bbP^1}(c,1), \mcO_{\bbP^1\times\bbP^1} (c+1, 1) \rangle
	\end{equation*}
	for some $c\in \bbZ$; see \cite[Cor.~4.27]{kuznetsov_exceptional_collections_in_surface_like_categories}.
	The corresponding Gram matrix is
	\begin{equation*}
		D_c\coloneqq \begin{pmatrix}
			1 & 2 & 2c+2 & 2c+4 \\
			0 & 1 & 2c & 2c+2 \\
			0 & 0 & 1 & 2 \\
			0 & 0 & 0 & 1
		\end{pmatrix},
	\end{equation*}
	see \cite[Ex.~3.7]{kuznetsov_exceptional_collections_in_surface_like_categories}. Now we mutate the third and fourth basis vector and compute the corresponding Gram matrices:
	\begin{align*}
		\sfL _{3,4} (b_1, \dots, b_4)=  (b_1, b_2, -2b_3+b_4, b_3), &&  \begin{pmatrix}
			1 & 2 & -(2(c-1)+2) & 2(c-1)+4 \\
			0 & 1 & - 2(c-1) & 2(c-1)+2 \\
			0 & 0 & 1 & -2 \\
			0 & 0 & 0 & 1
		\end{pmatrix}, \\
		\sfR _{3,4} (b_1, \dots, b_4)=  (b_1, b_2, b_4, b_3-2b_4), &&\begin{pmatrix}
			1 & 2 & 2(c+1) + 2 & -( 2(c+1)+4) \\
			0 & 1 & 2(c+1) & -(2(c+1)+2) \\
			0 & 0 & 1 & -2 \\
			0 & 0 & 0 & 1
		\end{pmatrix}.
	\end{align*}
	Multiplying $-2b_3 +b_4$ by $-1$ in the first case and $b_3-2b_4$ in the second case, we observe that all bases corresponding with Gram matrices $D_c$ are related by mutations up to sign and isometry.
\end{proof}

For later use in the proof of \cref{thm:main_result}, we also treat the surfaces $X_1$ and $X_2$ by hand. 
\begin{proposition}\label{prop:result_for_X_1}
	Any two exceptional bases in $\sfK_0^{\mathrm{num}}(X_1)$ are related by mutations up to sign and isometry.
\end{proposition}
\begin{proof}
	Since $\sfK_0^{\mathrm{num}}(X_1)$ is not isometric to $\sfK_0^{\mathrm{num}}(\bbP^1 \times \bbP^1)$, Perling's algorithm \cref{thm:perlings_algorithm} shows that a norm-minimal basis has the form $e_1, \dots, e_4$, where $e_1$ is of rank zero and $e_2, e_3$ and $e_4$ are of rank one. The contraction $G_{e_1}$ is isomorphic to $\sfK_0^{\mathrm{num}}(\bbP^2)$ with norm-minimal exceptional basis $e_2, e_3, e_4$. Since blow-up and blow-down are mutually inverse, $e_1$ results from blowing up a point $z= n \sfp \in \bbZ\sfp$.
	Observe that $\delta(G) = \delta(G_{e_1})=0$, so by \Cref{eq:degree_in_term_of_rank} the degree has to decrease by $1$ from $G$ to $G_{e_1}$. Thus by \Cref{lem:formula_degree_blow_up} $n=\pm 1$ and $\chi(\sigma, z)=\pm 1$. Possibly after changing the sign of $e_1$, the Gram matrix with respect to $e_1, \dots, e_4$ is $M_1$.
\end{proof}

The surface $X_2$ can be obtained from blowing up $\bbP^2$ in $2$ points or from blowing up $\bbP^1 \times \bbP^1$ in $1$ point. So a priori, there could potentially be two different types of norm-minimal exceptional bases. We compute that this is not the case.
\begin{proposition}\label{prop:blow_up_two_points_transitivity}
	Let $X_2$ be the blow-up of $\bbP^2$ in $2$ points and let $G_2\coloneqq \sfK_0 ^{\mathrm{num}}(X_2)$. Then any two exceptional bases are related by mutations up to sign and isometry. In particular, any norm-minimal exceptional basis is of norm $3$.
\end{proposition}
\begin{proof}
	We show that any exceptional basis can be mutated to an exceptional basis with Gram matrix $M_2$. Let $e_\bullet$ be an exceptional basis. Again with Perling's algorithm we mutate $e_\bullet$ to a norm-minimal basis $a_1, \dots, a_l, b_1 , \dots, b_m$ with $a_i$ of rank zero and $b_i$ of rank one. Now $m \in \{3,4\}$, since the (iterated) contraction of the rank zero elements yields a minimal geometric surface-like pseudolattice, which admits an exceptional basis; that implies it is isometric to $\sfK_0^{\mathrm{num}}(\bbP^2)$ or to $\sfK_0^{\mathrm{num}} (\bbP^1 \times \bbP^1)$.
	Assume for contradiction $m=4$. Then the contraction $G_{a_1}$ has Gram matrix 
	\begin{equation*}
		D_c\coloneqq \begin{pmatrix}
			1 & 2 & 2c+2 & 2c+4 \\
			0 & 1 & 2c & 2c+2 \\
			0 & 0 & 1 & 2 \\
			0 & 0 & 0 & 1
		\end{pmatrix}
	\end{equation*}
	with respect to $b_1, \dots, b_4$, see \cite[Ex.~3.7]{kuznetsov_exceptional_collections_in_surface_like_categories}. By \Cref{prop:result_for_P1_x_P1} we can assume that $c=0$.
	Moreover, as in the proof of \Cref{prop:result_for_X_1}, $a_1$ is obtained by blowing up $\sfK_0^{\mathrm{num}} (\bbP^1 \times \bbP^1)$ in $\pm \sfp$. After possibly changing the sign of $a_1$, we can assume $a_1$ results from blowing up $-\sfp$. Now we want to find a sequence of mutations, which reduces the norm of $(a_1, b_1, \dots, b_4)$. We compute:
	\begin{align}\label{eq:mutations_hirzebruch_iso}
		\begin{split}
			&\ (a_1, b_1, b_2, b_3, b_4) \\
			\xrightarrow{\sfL_{1,2}} &\ (-a_1+b_1, a_1, b_2, b_3, b_4)  \\
			\xrightarrow{\sfR_{2,3}} &\ (-a_1+b_1, b_2, a_1-b_2, b_3, b_4) \\
			\xrightarrow{\sfR_{3,4}} &\ (-a_1+b_1, b_2, b_3, a_1-b_2 -b_3, b_4)  \\
			\xrightarrow{\sfL_{1,2}} &\ (a_1-b_1+b_2, -a_1 +b_1, b_3, a_1-b_2 -b_3, b_4)  \\
			\xrightarrow{\sfL_{2,3}} &\ ( a_1-b_1+b_2, a_1-b_1+b_3, -a_1+b_1, a_1-b_2-b_3, b_4 )  \\
			\xrightarrow{\sfR_{4,5}} &\ (a_1-b_1+b_2, a_1-b_1+b_3, -a_1+b_1, b_4, a_1-b_2-b_3+3b_4).
		\end{split}
	\end{align}
	Since the rank map is additive one easily computes that the last basis is of rank $(0,0,1,1,1)$. But this contradicts the assumption that $(a_1, b_1, \dots, b_4)$ was norm-minimal. Thus $m=3$ and the exceptional basis $a_1, a_2, b_1, b_2, b_3$ results from blowing up $\sfK_0^{\mathrm{num}}(\bbP^2)$ in $2$ points $n_1\sfp$ and $n_2 \sfp$. After possibly changing signs, we can assume $n_1, n_2 \leq 0$.
	The fact that $G_2$ and $(G_2)_{a_1, a_2} = G_0$ have defect zero implies that also $(G_2)_{a_1}$ has defect zero, since contraction only increases the defect by \Cref{lem:formula_defect_contraction}. Therefore the degree has to increase by $1$ in each contraction and we have $n_1=n_2=-1$ by \Cref{lem:formula_degree_blow_up}. Hence, the Gram matrix with respect to to $a_1, a_2, b_1, b_2, b_3$ is $M_2$.
\end{proof}
\begin{remark}
	One can further compute the Gram matrix with respect to~$(a_1-b_1+b_2, a_1-b_1+b_3, -a_1+b_1, b_4, a_1-b_2-b_3+3b_4)$ as
	\begin{equation*}
		\begin{pmatrix}
			1 & 0 & -1 & -1 & -1 \\
			0 & 1 & -1 & -1 & -1 \\
			0 & 0 & 1 & 3 & 6   \\
			0 & 0 & 0 & 1 & 3  \\
			0 & 0 & 0 & 0 & 1
		\end{pmatrix},
	\end{equation*}
	but this will not be used subsequently.
\end{remark}

\begin{proof}[Proof of \Cref{thm:main_result}]
	We show the following statement: Given any exceptional basis $a_\bullet$ of $G_k$ we can find another exceptional basis related by mutations and sign changes to $a_\bullet$ such that the Gram matrix is of the form $M_k$ and the first $k$ basis elements have rank zero and the last $3$ have rank one. The unimodularity then ensures that the involved isometry preserves the surface-like structure given by $\sfp$, since the isometry respects the rank function. As we have treated the cases $k\leq 2$ by hand, we can assume $k>2$.	
	Given any exceptional basis of $G_k$ we can mutate the basis to a norm-minimal basis $$(a_1, \dots, a_l, b_1, \dots, b_m)$$ where the elements $a_i$ are of rank zero and the $b_i$ are of rank one and $m$ is equal to $3$ or $4$; see \cref{thm:perlings_algorithm}.
	Contracting the rank zero objects $a_i$, we obtain a minimal unimodular geometric surface-like pseudolattice $(G_k)_{a_1, \dots, a_l}$ admitting an exceptional basis. Thus the defect of $(G_k)_{a_1, \dots, a_l}$ is zero by \cite[Cor.~5.6]{kuznetsov_exceptional_collections_in_surface_like_categories}.
	Contraction of geometric pseudolattices only increases the defect, cf.~\Cref{lem:formula_defect_contraction}, thus all intermediate pseudolattices $(G_k)_{a_1, \dots, a_i}$ are unimodular geometric surface-like pseudolattices with defect zero and admit an exceptional basis. This implies that they are isometric to blow-ups of $\bbP^2$ as long as $k-i \geq 2$ by \Cref{thm:classification}.
	Choosing $i$ such that $k-i =2$, the pseudolattice $(G_k)_{a_1, \dots, a_i}$ is isometric to the blow-up of $2$ points in $(G_k)_{a_1, \dots, a_i}$. By \Cref{prop:blow_up_two_points_transitivity} any two exceptional bases of $(G_k)_{a_1, \dots, a_i}$ are related by mutations up to sign and we can mutate the exceptional basis to a basis of norm $3$.
	Now mutations in the contraction lift to mutations of $G_k$, which leave the contracted vectors invariant. Hence, $m=3$ and $l=k$. In particular, $(G_k)_{a_1, \dots, a_k}$ is isometric to $G_0$ and we may assume that $b_1,b_2, b_3$ have Gram matrix $M_0$.
	
	As seen in \Cref{se:blow_up_down}, blowing up and contracting are mutually inverse operations. Thus the basis $a_1, \dots, a_k, b_1, b_2, b_3$ is a basis obtained from blowing up $G_0$ in $k$ points.
	The point-like element of $G_0$ is unique up to sign, as discussed in \cite[Ex.~3.5]{kuznetsov_exceptional_collections_in_surface_like_categories}, hence we can assume $\sfp=b_3-2b_2+b_1$.
	In each intermediate step $(G_k)_{a_1, \dots, a_{i+1}}$ is obtained from $(G_k)_{a_1, \dots, a_i}$ by blowing up a point $n_{i+1}\sfp$ with $n_{i+1} \in \bbZ$.
	As each $(G_k)_{a_1, \dots, a_{j}}$ has defect zero we deduce $n_j = \pm 1$ for all $j$. Indeed, by \Cref{eq:degree_in_term_of_rank} the degree has to decrease by $-1$ in each step and \Cref{lem:formula_degree_blow_up} yields
	\begin{equation*}
		\deg ((G_k)_{a_1, \dots, a_{i}}) = \deg ((G_k)_{a_1, \dots, a_{i+1}})- \chi(\sigma, n_{i+1} \sfp )^2 = \deg ((G_k)_{a_1, \dots, a_{i+1}}) - n_{i+1}^2.
	\end{equation*}
	Up to possibly changing signs, we can arrange $\chi(a_i, b_j)=1$ for all $i, j$. Thus the Gram matrix has the desired form.
\end{proof}

\section{Blow-up of 9 Points}\label{se:9_points}
Full exceptional collections on del Pezzo surfaces were studied in \cite{kuleshov_orlov_exceptional_sheaves_on_del_pezzo_surfaces} and in \cite{elagin_lunts_on_full_exceptional_collections_of_line_bundles_on_del_pezzo_surfaces}. In \cite{elagin_xu_zhang_on_cyclic_strong_exceptional_collections_of_line_bundles_on_surfaces} and \cite{ishi_okawa_uehara_exceptional_collections_on_sigma_2} similar results for \emph{weak del Pezzo surfaces}, i.e.~surfaces with nef and big anticanonical divisor, were obtained. 
In this section, we expand the class of examples by considering the blow-up of $\bbP^2$ in $9$ points in very general position.
In this situation, we can assume that there is a unique cubic curve in $\bbP^2$ passing through each of the $9$ points with multiplicity $1$. Then the divisor class of the strict transform of this cubic coincides with the anticanonical divisor $-K_X=3H-\sum_{i=1}^9 E_i$ of the blow-up $X$. Here $H$ is the pullback of a hyperplane class in $\bbP^2$ and $E_i$ is the exceptional divisor corresponding to the blow-up of the point $p_i$.
Therefore, $-K_X$ is nef but not big as $(-K_X)^2=0$, so $X$ is not a weak del Pezzo surface. Additionally, $-K_X$ is not basepoint-free and for that reason the techniques developed in \cite{kuleshov_exceptional_and_rigid_sheaves_on_surfaces_with_anticaonical_class_without_base_components} cannot be applied. In this section we exclusively work over the field of complex numbers.

\subsection{Toric systems and numerically exceptional collections}
We recall the necessary terminology of toric systems as introduced by Hille and Perling in \cite[\S\S\,2-3]{hille_perling_exceptional_sequences_of_invertible_sheaves_on_rational_surfaces}.
\begin{definition}\label{def:toric_system}
	Let $X$ be a smooth projective surface. A sequence of divisors $A_1, \dots, A_n$ on $X$ is a \emph{toric system} if $n\geq 3$ and one has $A_i \cdot A_{i+1}=1=A_1 \cdot A_n$ for all $1\leq i \leq n-1$, $A_i \cdot A_j=0$ for $|i-j| >1$ except $\{i,j\}=\{1,n\}$, and $A_1 + \dots + A_n \sim_{\mathrm{lin}} -K_X$.
\end{definition}
If $\chi(\mcO_X)=1$ and $n=\rk \sfK_0(X)$, we have a 1:1-correspondence between toric systems on $X$ and numerically exceptional collections consisting of line bundles of length $n$ up to twists with line bundles:
\begin{equation*}
	\left\{  \begin{matrix}  \text{toric systems } (A_1, \dots, A_n)
		
	\end{matrix} \right\}  / \sim_{\mathrm{lin}} \:
	\leftrightarrow  \left\{  \begin{matrix}  \text{numerically exceptional collections }
		
		\\  \text{of line bundles } (\mcO_X(D_1), \dots, \mcO_X(D_n))
		
	\end{matrix} \right\} / \mathrm{Pic}(X)
\end{equation*}
A toric system $(A_1, \dots, A_n)$ and a choice of a divisor $D_1$ defines a numerically exceptional collection $(\mcO_X(D_1), \dots, \mcO_X(D_n))$ given by $$ D_{i+1}\coloneqq  D_1 + A_1 + \dots + A_i.$$
Conversely, any numerically exceptional collection of line bundles $(\mcO_X(D_1), \dots, \mcO_X(D_n))$ gives rise to a toric system via $$ A_i\coloneqq \begin{cases}
	D_{i+1}-D_i & \text{for}\; 1 \leq i \leq n-1, \\
	D_1-K_X-D_n &  \text{for} \; i =n.
\end{cases}$$

A toric system is called \emph{exceptional} if the corresponding collection of line bundles is exceptional. Equivalently, each divisor $A_i + \dots + A_j$ ($1\leq i\leq j\leq n-1$) is \emph{left-orthogonal} (a divisor $D$ is called left-orthogonal if $h^i(-D)=0$ for all $i$). Moreover, $(A_1, \dots, A_n)$ is an (exceptional) toric system if and only if $(A_2, \dots, A_n, A_1)$ is an (exceptional) toric system.

Orlov's blow-up formula for full exceptional collections can be transferred to toric systems via so-called \emph{augmentations}; see \cite[\S\,5]{hille_perling_exceptional_sequences_of_invertible_sheaves_on_rational_surfaces} and \cite[\S\,2.6]{elagin_lunts_on_full_exceptional_collections_of_line_bundles_on_del_pezzo_surfaces}: If $X'$ is a surface with toric system $A_1', \dots, A_n'$ and $p\colon X \to X'$ the blow-up of $X$ in a closed point $p\in X'$ with exceptional divisor $E\subseteq X$, denote by $A_i\coloneqq p^* A_i'$ the pullback of the divisors. We obtain a toric system on $X$, namely
$$E, A_1 -E, A_2, \dots, A_{n-1}, A_n-E.$$
This toric system and all its cyclic shifts are called \emph{augmentations}.
Conversely, a blow-down operation for toric systems can be defined.
\begin{proposition}[{\cite[Prop.~3.3]{elagin_lunts_on_full_exceptional_collections_of_line_bundles_on_del_pezzo_surfaces}}]
	Let $A_1, \dots, A_n$ be a toric system on a surface $X$ such that there exists an index $1\leq m \leq n$ with $A_m$ a $(-1)$-curve in $X$. Let $p\colon X \to X'$ be the blow-down of $A_m$. Then $A_1, \dots, A_n$ is an augmentation of a toric system $A_1', \dots, A_{n-1}'$ on $X'$.
\end{proposition}

An essential observation for the proof of \cref{thm:main_result_2} is that \cite[Lem.~3.4]{elagin_lunts_on_full_exceptional_collections_of_line_bundles_on_del_pezzo_surfaces} generalizes to the blow-up of $\bbP^2$ in $9$ points:
\begin{lemma}\label{lem:left_orth_divisor_-1}
	Let $X$ be the blow-up of $\bbP^2$ in $9$ points in very general position. Then any divisor $D$ with $D^2=-1$ and $\chi(D)=1$ is linearly equivalent to a $(-1)$-curve.
\end{lemma}
\begin{proof}
	First of all, Riemann--Roch yields $\chi(D)=1+\frac{1}{2}(D(D-K_X))$. As $D^2=-1$ we have $-K_XD=1$.
	Now $\chi(D)=h^0(D)-h^1(D)+h^2(D)$ and $h^2(D)=h^0(K_X-D)$ by Serre duality. The intersection $-K_X(K_X-D)=K_XD=-1$ implies that $K_X-D$ is not effective, since $-K_X$ is nef. Therefore $h^2(D)=h^0(K_X-D)=0$ and in order for $\chi(D)=1$ to be fulfilled, $D$ must have at least one nontrivial global section, i.e.~$D$ must be effective.
	We write $D=\sum_i k_i C_i$, where the $C_i$ are pairwise distinct integral curves in $X$ and the $k_i$ are positive integers.
	From the equation $1=-K_XD=\sum_ik_i (-K_X)C_i$ and the nefness of $-K_X$ we derive that among the curves $C_i$ there is one $C_0$ occurring with coefficient $1$ and satisfying $-K_XC_0=1$. All other $C_i$ lie in $K_X^\perp$. 
	Note that by \cite[Prop.~2.3]{de_fernex_negative_curves_on_very_general_blow_ups_of_p2} any integral curve with negative self-intersection is a $(-1)$-curve. Therefore no curve in $K_X^\perp$ can have negative self-intersection, as for $(-1)$-curves the intersection with the canonical class is nonzero.
	Hence, in order to achieve $-1=D^2$ we must have $C_0^2=-1$, $C_i^2=0$ and $C_0C_i=0$ for all $i\neq 0$. Let $A\coloneqq D-C_0$. Then $A \in K_X^\perp$ and $A^2=0$. But this implies $A=nK_X$, since any isotropic vector in $K_X^\perp$ is a multiple of $K_X$. Now $C_0K_X=-1$ together with $C_0 A=0$ implies $n=0$ and hence $C_0=D$.
\end{proof}
\begin{remark}[On the position of the blown up points]\label{rmk:pos_of_9_points}
	The position of the $9$ blown up points is important for only two facts: On the one hand we need to choose the position general enough so that there exists a unique cubic passing through the $9$ points with multiplicity $1$ and on the other hand in the proof of \cref{lem:left_orth_divisor_-1} we use the result of \cite{de_fernex_negative_curves_on_very_general_blow_ups_of_p2} which depends on the position of the points. For the latter the assumptions are made more precise in \cite[Def.~2.1]{de_fernex_negative_curves_on_very_general_blow_ups_of_p2}.
	Alternatively one can replace \cite[Prop.~2.3]{de_fernex_negative_curves_on_very_general_blow_ups_of_p2} by \cite[Prop.~12]{nagata_on_rationl_surface_ii} and assume that the points are in a position described in \cite[Prop.~9]{nagata_on_rationl_surface_ii} to ensure that the surface carries no integral curve $C$ with $C^2\leq -2$.
\end{remark}
\begin{remark}
	We will see in \cref{sec:blow_up_10_points} that the conclusion of \cref{lem:left_orth_divisor_-1} does not hold for blow-ups of $10$ or more points.
\end{remark}

\subsection{Roots in the Picard lattice}\label{rmk:weyl_group}
	Recall that any unimodular lattice $\Lambda$ contains a root system with roots given by the elements $\alpha \in \Lambda$ such that $\alpha^2=\pm 1$ or $\alpha^2=\pm 2$. For such a root $\alpha \in \Lambda$ the reflection along $\alpha ^\perp$ is given by
	$$s_\alpha (x)\coloneqq  x - 2\frac{(x \cdot \alpha)}{\alpha^2}\alpha.$$
	Any such reflection is an orthogonal transformation of $\Lambda$.
		
	Let $X$ be the blow-up of $\bbP^2$ in $n\geq 3$ points. Let $H$ be the pullback of a hyperplane class and let $E_1, \dots, E_n$ be the exceptional divisors. Then $H, E_1, \dots, E_n$ is an orthogonal basis of the Picard lattice $\Pic(X)$ such that $H^2=1$ and $E_i^2=-1$.
	The elements $\alpha_1\coloneqq E_1 - E_2, \dots , \alpha_{n-1}\coloneqq E_{n-1} - E_n$ and $\alpha_0\coloneqq H-E_1-E_2-E_3$ are roots in $\Pic(X)$ and we denote by $W_X$ the reflection group generated by $s_{\alpha_i}$, $0\leq i\leq n-1$.
	All roots $\alpha_i$ lie in $K_X^\perp$, thus
	$$W_X \subseteq \Orth(\Pic(X))_{K_X} \subseteq \Orth(\Pic(X)),$$
	where $\Orth(\Pic(X))_{K_X}$ is the stabilizer of the canonical class $K_X = -3H+\sum_i E_i$.
	\begin{lemma}\label{lem:weyl_group_is_stabilizer}
		Let $X$ be the blow-up of $\bbP^2$ in $n$ points, where $3 \leq n \leq 9$. Then the reflection group $W_X=\langle s_{\alpha_0}, \dots, s_{\alpha_{n-1}} \rangle$ equals the stabilizer $\Orth(\Pic(X))_{K_X}$.
	\end{lemma}
	\begin{proof}
		First note that the equality $W_X = \Orth(\Pic(X))_{K_X}$ does not depend on the position of points.
		Thus, we can assume that the points lie in very general position and $-K_X$ is class of an irreducible reduced curve in $X$.
		Then \cref{lem:left_orth_divisor_-1} (or \cite[Lem.~3.4]{elagin_lunts_on_full_exceptional_collections_of_line_bundles_on_del_pezzo_surfaces} if $n\leq 8$) implies that any orthogonal transformation in $\Orth(\Pic(X))_{K_X}$ maps $(-1)$-curves to $(-1)$-curves.
		Hence, by \cite[Thm.~0.1]{harbourne_blowings_up_of_p2_and_their_blowings_down}, which is essentially a reformulation of results in \cite{nagata_on_rationl_surface_ii}, any such transformation is an element of $W_X$ and thus the lemma holds.
	\end{proof}

\subsection{A weak del Pezzo surface admitting a numerically exceptional collection of maximal length which is not exceptional}\label{se:example_not_full_exceptional}
We cannot expect that the conclusion of \cref{lem:left_orth_divisor_-1} holds true for rational surface of higher Picard rank, as we show in \cref{sec:blow_up_10_points}. But already if we blow up less than $9$ points in special position, the conclusion of \cref{lem:left_orth_divisor_-1} does not hold. As a consequence, in general a maximal numerically exceptional collection does not need to be exceptional. In \cref{ex:special_position} we construct such an example by blowing up $8$ points in a special position. Similar examples were already obtained for Hirzebruch surfaces $\Sigma_d$ with even $d$ in \cite[Rmk.~2.18]{elagin_lunts_on_full_exceptional_collections_of_line_bundles_on_del_pezzo_surfaces}.
\begin{proposition}\label{ex:special_position}
	Let $\pi \colon X \to \bbP^2$ be the blow-up of $8$ points $p_1, \dots, p_8$ such that $p_1, p_2, p_3$ lie on a line $L$ and $p_4, \dots, p_8$ on a smooth irreducible conic curve $C$ such that $p_1, p_2, p_3\notin C$ and $p_4, \dots, p_8 \notin L$. Then $X$ is a weak del Pezzo surface, i.e.~$-K_X$ is nef and big, but admits a maximal numerically exceptional collection consisting of line bundles which is not exceptional. Moreover, $X$ admits an effective divisor $D$ satisfying $D^2=-1$, $\chi(D)=1$ and $H^1(X, \mcO_X(-D))\neq 0$.
\end{proposition}
\begin{proof}
	Denote by $E_i$ the exceptional divisor corresponding to the blow-up of $p_i$ and let $H$ be the pullback of the hyperplane class in $\bbP^2$. Then the anticanonical divisor satisfies $-K_X=3H-\sum_{i=1}^{8}E_i$ and thus is equal to the sum of the strict transform $\tilde{L}$ of $L$ and the strict transform $\tilde{C}$ of $C$. Hence, the intersection of $-K_X$ with any other curve is non-negative and one checks that $-K_X\tilde{L}=0$ is zero and $-K_X\tilde{C}=1$. Therefore $-K_X$ is nef and hence $(-K_X)^2=1>0$ implies that $-K_X$ is big.
	Consider the divisor $$D\coloneqq 4H-2E_1-2E_2-2E_3-E_4-\dots -E_8,$$ which satisfies $D^2=-1$ and $-K_XD=1$.
	Arguing as in the proof of \cref{lem:left_orth_divisor_-1} one observes that $D$ is effective. But $D$ is not irreducible, which can be seen as follows: Assume for contradiction that $D$ is irreducible. Since $D$ is not one of the exceptional divisors $E_i$, $D$ must be the strict transform of a curve $B$ in $\bbP^2$.
	Now by B\'{e}zout's theorem $B\cdot L=\deg(B)\deg(L)=\deg(B)$ or $B=L$. The latter cannot occur since $\tilde{L}$ has class $H-E_1-E_2-E_3 \neq D$.
	By the explicit form of $D$, we must have $\deg(B)=4$ and the multiplicity of $B$ at $p_i$ must be $2$ for $i=1,2,3$. Hence, $B\cdot L\geq 6$, which contradicts to $B\cdot L=\deg(B)$. Thus, $D$ cannot be irreducible.
	Further we compute $h^1(-D)=1$: Riemann--Roch yields $\chi(-D)=0$ and since $D$ is effective, $-D$ admits no global sections. This gives $h^1(-D)=h^2(-D)=h^0(K_X+D)=h^0(H-E_1-E_2-E_3)=1$. Thus the conclusion of \cref{lem:left_orth_divisor_-1} does not hold for $D$.
	
	Finally, we complete $D$ into a toric system in order to obtain a maximal numerically exceptional collection consisting of line bundles.
	The set of orthogonal transformations of $\Pic(X)$ fixing the canonical class $K_X$ coincides with the orthogonal group of $K_X^\perp$. A computation shows that $K_X^\perp$ identifies with the $E_8$-lattice and therefore the orthogonal group is the Weyl group of $E_8$. It is known that the Weyl group acts transitively on the set of roots; see, e.g., \cite[\S\,10.4 Lem.~C]{humphreys_intoduction_to_lie_algebras_and_representation_theory}. Further, we can write the exceptional divisor as $E_1=(K_X+E_1)-K_X$ and compute $(K_X+E_1)^2=-2$. Hence, there is an orthogonal transformation $T$ fixing $K_X$ and sending the root $K_X+E_1$ to the root $T(K_X+E_1)=H-E_1-E_2-E_3$. Thus, $E_1$ is mapped to $D=H-E_1-E_2-E_3-K_X$ under $T$. Therefore, the image of the toric system associated to
	$$\sfD^b(X)=\langle \mcO_X, \mcO_X(E_1), \dots, \mcO_X(E_8), \mcO_X(H), \mcO_X(2H) \rangle$$
	under $T$ is a toric system which corresponds to a maximal numerically exceptional collection consisting of line bundles, which is not exceptional since $H^1(X, \mcO_X(D))\neq 0$.
\end{proof}

\subsection{Towards \texorpdfstring{\cref{thm:main_result_2}}{Theorem \ref{thm:main_result_2}}}
The proof of \cref{thm:main_result_2} is separated in two steps. Recall that \cite[Thm.~3.1]{elagin_lunts_on_full_exceptional_collections_of_line_bundles_on_del_pezzo_surfaces} states that, on a del Pezzo surface, any toric system is obtained from a sequence of augmentations from an exceptional toric system on $\bbP^2$ or a Hirzebruch surface.
In the first step, we generalize this result to the blow-up $X$ of $9$ points in very general position.
In the second step, we prove the transitivity of the braid group action, as stated in \cref{thm:main_result_2}, by realizing each orthogonal transformation of $\Pic(X)$ fixing $K_X$ as a sequence of mutations.

The following \cref{lem:blow_down_del_pezzo} ensures that we can reduce $X$ to a del Pezzo surface by contracting any $(-1)$-curve. As we were unable to find a suitable statement in the literature, we include a proof.
\begin{lemma}\label{lem:blow_down_del_pezzo}
	Let $X$ be the blow-up of $\bbP^2$ in $9$ points in very general position and let $E\subseteq X$ be a $(-1)$-curve. Then the surface $Y$ obtained from blowing down $E$ is a del Pezzo surface.
\end{lemma}
\begin{proof}
	Recall that the blow-up of less than $8$ points in $\bbP^2$ is a del Pezzo surface if and only if not $3$ of the points lie on a line and not $6$ lie on a conic; see, e.g., \cite[Thm.~24.4]{manin_cubic_forms}. Therefore the points are in special position if and only if the surface admits a $(-2)$-curve, namely the strict transform of a conic through $6$~blown up points or the line through $3$~blown up points.
	We further observe that the equivalence also holds true if the points are chosen infinitely near: If a point $p$ is blown up on an exceptional divisor $E$, then the class of the strict transform of $E$ is $E-E_p$, where $E_p$ is the exceptional divisor corresponding to the blow-up of $p$. We compute $(E-E_p)^2=-2$ in that case.
	
	Let $Y$ be the blow-down of the $(-1)$-curve and $\pi \colon X \to Y$ the blow-up map with center $p \in Y$ and exceptional divisor $E$. Then for any curve $C$ in $Y$, the strict transform in $X$ has divisor class $p^*C - mE$, where $m$ is the multiplicity of $C$ at $p$. Thus the self-intersection of the strict transform of $C$ is $C^2-m$. Hence, if $X$ has no $(-2)$-curves, then $Y$ has no $(-2)$-curves. Now $Y$ is obtained from $\bbP^2$ by a sequence of blow-ups of (possibly infinitely near) $8$ points. As $Y$ has no $(-2)$-curves, $Y$ must be a del Pezzo surface.
\end{proof}
\begin{theorem}\label{thm:generalization_elagin_lunts}
	Let $X$ be the blow-up of $\bbP^2$ in $9$ in very general position.
	Any toric system on $X$ is a standard augmentation, i.e.~it is obtained by a sequence of augmentations from a full exceptional toric system on $\bbP^2$ or from a full exceptional toric system on a (non necessarily minimal) Hirzebruch surface.
\end{theorem}
\begin{proof}
	Let $A_1, \dots, A_{12}$ be a toric system on $X$. By \cref{lem:left_orth_divisor_-1}, \cref{lem:blow_down_del_pezzo}, and \cite[Thm.~3.1]{elagin_lunts_on_full_exceptional_collections_of_line_bundles_on_del_pezzo_surfaces} we only need to show that there is a divisor $A_i$ with $A_i^2=-1$. In this situation the argument of Elagin--Lunts still applies: By \cite[Prop.~2.7]{hille_perling_exceptional_sequences_of_invertible_sheaves_on_rational_surfaces} there exists a smooth toric surface $Y$ with torus invariant irreducible divisors $D_1, \dots, D_{12}$ such that $D_i^2 = A_i^2$ for any $i$. Since $Y$ is not minimal, $Y$ contains $(-1)$-curve which must be torus invariant as otherwise the self-intersection would be non-negative. We conclude that one of the $D_i$ squares to $-1$, hence there exists $A_i$ with $A_i^2=-1$.
\end{proof}
\begin{corollary}\label{cor:generalization_elagin_lunts}
	On the blow-up of $\bbP^2$ in $9$ very general points any numerically exceptional collection of maximal length consisting of line bundles is a full exceptional collection.
\end{corollary}
\begin{proof}
	By \cite[Prop.~2.21]{elagin_lunts_on_full_exceptional_collections_of_line_bundles_on_del_pezzo_surfaces} a standard augmentation corresponds to a full exceptional collection.
\end{proof}
In order to conclude the proof of \cref{thm:main_result_2} we are left to show that any two full exceptional collections resulting from two different sequences of augmentations are related by mutations and shifts. On a del Pezzo surface, an exceptional object is completely determined by its class in the Grothendieck group:
\begin{lemma}[Exceptional objects on del Pezzo surfaces, \cite{gorodentsev_exceptional_bundles_on_surfaces_with_a_moving_anticanonical_class, kuleshov_orlov_exceptional_sheaves_on_del_pezzo_surfaces}]\label{lem:exc_obj_del_pezzo}
	Let $X$ be a del Pezzo surface and let $E \in \sfD^b(X)$ be an exceptional object. Then $E$ is isomorphic to some $F[k]$, where $F$ is an exceptional sheaf on $X$ and $k \in \bbZ$. Moreover $F$ is either locally free or a torsion sheaf of the form $\mcO_C(d)$, where $C$ is a $(-1)$-curve. In particular, two exceptional objects with the same image in $\sfK_0(X)$ only differ by an even number of shifts.
\end{lemma}
\begin{proof}[Pointer to references]
	That every exceptional object is a sheaf up to shift can be found in \cite[Prop.~2.10]{kuleshov_orlov_exceptional_sheaves_on_del_pezzo_surfaces} and \cite[Prop.~2.9]{kuleshov_orlov_exceptional_sheaves_on_del_pezzo_surfaces} states that an exceptional sheaf is locally free or a torsion sheaf of the form $\mcO_C(d)$ where $C$ is a $(-1)$-curve. In the latter case, such torsion sheaf is clearly uniquely determined by their Chern character and hence by their class in $\sfK_0(X)$. The case of locally free sheaves is treated in \cite[Cor.~2.5]{gorodentsev_exceptional_bundles_on_surfaces_with_a_moving_anticanonical_class}.
\end{proof}

For later use in the proof of \cref{thm:transitivity_on_9_points} we compute in the following \cref{lem:two_blow_up_realizations_are_related_by_mutations} a relation by mutations and shifts of two concrete exceptional collections on the blow-up of $\bbP^2$ in $3$ points. The statement of \cref{lem:two_blow_up_realizations_are_related_by_mutations} can also be deduced from \cite[Thm.~7.7]{kuleshov_orlov_exceptional_sheaves_on_del_pezzo_surfaces}. We give an independent proof by computing an explicit sequence of mutations relating both collections.
\begin{lemma}\label{lem:two_blow_up_realizations_are_related_by_mutations}
	Let $X$ be the blow-up of $3$ points in $\bbP^2$ which do not lie on a line. Then the full exceptional collections $$\sfD^b(X)=\langle \mcO_{E_1}(-1),\mcO_{E_2}(-1),\mcO_{E_3}(-1), \mcO_X, \mcO_X(H), \mcO_X(2H) \rangle$$ and
	\begin{align*}
		\sfD^b(X) =\langle & \mcO_{H-E_2-E_3}(-1), \mcO_{H-E_1-E_3}(-1), \mcO_{H-E_1-E_2}(-1),\\
		& \mcO_X, \mcO_X(2H-E_1-E_2-E_3), \mcO_X(4H-2E_1-2E_2-2E_3) \rangle
	\end{align*}
	are related by mutations and shifts.
\end{lemma}
\begin{proof}
	Since $X$ is a del Pezzo surface it is enough to verify the claim in $\sfK_0(X)$ by using \cref{lem:exc_obj_del_pezzo}. In $\sfK_0(X)$ this becomes a lattice-theoretic computation:
	Let $$a_i\coloneqq [\mcO_{E_i}(-1)] \; \text{and} \; b_1\coloneqq [\mcO_X], b_2\coloneqq [\mcO_X(H)], b_3\coloneqq [\mcO_X(2H)].$$
	Then the Gram matrix corresponding to the basis $(a_1, a_2, a_3, b_1, b_2, b_3)$ is
	\begin{equation}\label{eq:gram_matrix_blow_up_3_points}
		\begin{pmatrix}
			1 & 0 & 0 & -1 & -1 & -1 \\
			0 & 1 & 0 & -1 & -1 & -1 \\
			0 & 0 & 1 & -1 & -1 & -1 \\
			0 & 0 & 0 & 1 & 3 & 6   \\
			0 & 0 & 0 & 0 & 1 & 3  \\
			0 & 0 & 0 & 0 & 0 & 1
		\end{pmatrix}.
	\end{equation}
	Similarly to \cref{eq:mutations_hirzebruch_iso} we have the following sequence of mutations
	\begin{align*}
		(a_1, a_2, a_3, b_1, b_2, b_3) \xrightarrow{\sfL_{5,6}}
		\xrightarrow{\sfR_{3,4}} 
				\xrightarrow{\sfR_{2,3}}
				\xrightarrow{\sfL_{4,5}} 
				\xrightarrow{\sfL_{3,4}} 
				\xrightarrow{\sfR_{2,3}} 
				\xrightarrow{\sfR_{1,2}} 
				\xrightarrow{\sfL_{2,3}}
				\xrightarrow{\sfR_{3,4}} 
				\xrightarrow{\sfR_{4,5}}
				\xrightarrow{\sfL_{2,3}} 
				\xrightarrow{\sfL_{3,4}}
				\xrightarrow{\sfR_{5,6}} \\
				 (a_2+a_3+2b_1-3b_2+b_3, -a_1-a_3-2b_1+3b_2-b_3, -a_1-a_2-2b_1+3b_2-b_3,\\
				  a_1+a_2+a_3+3b_1-3b_2+b_3, b_2, a_1+a_2+a_3+2b_1-3b_2 ).
	\end{align*}
	After changing the sign of the first and last basis elements we obtain the exceptional basis
	\begin{align}\label{eq:intermediate_computation_f_e_c_2}
		(-a_2-a_3-2b_1+3b_2-b_3, -a_1-a_3-2b_1+3b_2-b_3, -a_1-a_2-2b_1+3b_2-b_3,\\
		 a_1+a_2+a_3+3b_1-3b_2+b_3, b_2, -a_1-a_2-a_3-2b_1+3b_2 ),\nonumber
	\end{align}
	which has Gram matrix \Cref{eq:gram_matrix_blow_up_3_points}.
	Recall that the Chern character on a surface is given by $$\ch=\left(\rk,\cc_1,\frac{1}{2}(\cc_1^2-2\cc_2)\right).$$
	We compute
	\begin{align*}
		\ch(a_i) &{}=\ch(\mcO_{E_i} (E_i))=\left(0,E_i, -\frac{1}{2}\right),\\
		\ch(b_1) &{}=\ch(\mcO_X)=(1,0,0),\\
		\ch(b_2) &{}=\ch(\mcO_X(H))=\left(1,H,\frac{1}{2}\right),\\
		\ch(b_3) &{}=\ch(\mcO_X(2H))=(1,2H,2).
	\end{align*}
	Thus \cref{eq:intermediate_computation_f_e_c_2} corresponds to the full exceptional collection
	\begin{equation}\label{eq:intermediate_computation_f_e_c}
		\langle  \mcO_{H-E_2-E_3}, \mcO_{H-E_1-E_3}, \mcO_{H-E_1-E_2},
		 \mcO_X(-H+E_1+E_2+E_3), \mcO_X(H), \mcO_X(3H-E_1-E_2-E_3) \rangle.
	\end{equation}
	We observe that $H-E_i-E_j$ is the class of the strict transform of the line through the points $p_i$ and $p_j$ and $K_X$ can be rewritten as
	\begin{align*}
		K_X & { }=-3H+E_1+E_2+E_3\\
		 & { }=-3(2H-E_1-E_2-E_3)+(H-E_2-E_3)+(H-E_1-E_3)+(H-E_1-E_2),
	\end{align*}
	where $2H-E_1-E_2-E_3$ can be identified with the pullback of a hyperplane class on $\bbP^2$ considered as the blow-down of $(H-E_2-E_3)$, $(H-E_1-E_3)$ and $(H-E_1-E_2)$. Hence $$ \mcO_{H-E_i-E_j}(K_X)=\mcO_{H-E_i-E_j}(H-E_i-E_j)=\mcO_{H-E_i-E_j}(-1),$$
	where we have used the projection formula in the first equality.
	Recall that any twist with an integer multiple of the canonical line bundle can be realized as a sequence of mutations. Twisting \Cref{eq:intermediate_computation_f_e_c} by $K_X$ yields
	\begin{align*}
		\langle & \mcO_{H-E_2-E_3}(-1), \mcO_{H-E_1-E_3}(-1), \mcO_{H-E_1-E_2}(-1), \\
		& \mcO_X(-4H+2E_1+2E_2+2E_3), \mcO_X(-2H+E_1+E_2+E_3), \mcO_X \rangle.
	\end{align*}
	Finally, by applying the sequence $\sfR_{5,6} \circ \sfR_{4,5} \circ \sfR_{5,6} \circ \sfR_{4,5}$ of mutations, we obtain the desired full exceptional collection
	\begin{align*}
		\sfD^b(X) =\langle & \mcO_{H-E_2-E_3}(-1), \mcO_{H-E_1-E_3}(-1), \mcO_{H-E_1-E_2}(-1),\\
		& \mcO_X, \mcO_X(2H-E_1-E_2-E_3), \mcO_X(4H-2E_1-2E_2-2E_3) \rangle. \qedhere
	\end{align*}
\end{proof}
\begin{remark}[Geometric interpretation of \cref{lem:two_blow_up_realizations_are_related_by_mutations}]
	The surface $X$ in \cref{lem:two_blow_up_realizations_are_related_by_mutations} admits two different blow-up realizations. First one blows up $3$ points $p_1, p_2, p_3$ in $\bbP^2$ and then one contract the $(-1)$-curves $H-E_i-E_j$ which are the strict transforms of the lines through the points $p_i, p_j$. The full exceptional collections compared in \Cref{lem:two_blow_up_realizations_are_related_by_mutations} are the collections resulting from Orlov's blow-up formula applied to these different realizations of $X$.
	Moreover, this construction defines a birational map $\bbP^2 \dashrightarrow \bbP^2$, which is known as \emph{standard quadratic Cremona transformation}.
\end{remark}

\begin{theorem}\label{thm:transitivity_on_9_points}
	On the blow-up $X$ of $\bbP^2$ in $9$ points in very general position, any two full exceptional collections consisting of line bundles are related by mutations and shifts.
\end{theorem}
\begin{proof}
	For the sake of simplicity we call two full exceptional collections \emph{equivalent} if they can be transformed into each other by a sequence of mutations and shifts.
	Let
	\begin{equation}\label{eq:f_e_c_start}
		\sfD^b(X)=\langle \mcO_X(D_1), \dots, \mcO_X(D_n)\rangle
	\end{equation}
	be a full exceptional collection consisting of line bundles, then $\langle \mcO_X(D_2), \dots, \mcO_X(D_n), \mcO_X(D_1-K_X)\rangle$ is an equivalent collection. In particular, any twist with an integer multiple of the canonical class can be realized as a sequence of mutations and shifts.
	By \cref{thm:generalization_elagin_lunts} the toric system associated to \cref{eq:f_e_c_start} contains a $(-1)$-curve. After passing to an equivalent collection, we can assume that $E\coloneqq D_2-D_1$ is a $(-1)$-curve. The left mutation of the pair $\langle\mcO_X(D_1), \mcO_X(D_2)\rangle$ is defined by the exact triangle
	\begin{equation*}
		\mcO_X(D_1)\otimes \RHom(\mcO_X(D_1),\mcO_X(D_2))\xrightarrow{\mathrm{ev}} \mcO_X(D_2) \to \sfL_{\mcO_X(D_1)}(\mcO_X(D_2)).
	\end{equation*}
	On the other hand
	\begin{equation*}
		\RHom(\mcO_X(D_1), \mcO_X(D_2))=H^\bullet(X, \mcO_X(E))=\bbC[0].
	\end{equation*}
	Therefore the ideal sheaf sequence $$0 \to \mcO_X(-E)\to \mcO_X\to \mcO_E\to0$$ yields an exact triangle
	$$\mcO_X(D_1)\otimes \RHom(\mcO_X(D_1), \mcO_X(D_2)) \to \mcO_X(D_2)\to \mcO_E(D_2).$$
	As $E$ is isomorphic to a projective line, we conclude that \cref{eq:f_e_c_start} is equivalent to
	\begin{equation}\label{eq:f_e_c_2}
		\sfD^b(X)=\langle \mcO_E(d), \mcO_X(D_1), \mcO_X(D_3), \dots, \mcO_X(D_n)\rangle
	\end{equation}
	for some $d \in \bbZ$.
	Let $p\colon X \to X'$ be the blow-down of $E$; then $K_X=p^*K_{X'}+E$. Using the projection formula to compute $\mcO_E(d) \otimes \mcO_X(K_X)=\mcO_E(d) \otimes \mcO_X(E)$, we can assume that $d=-1$ by twisting \cref{eq:f_e_c_2} with $(d-1)K_X$.
	This means that \cref{eq:f_e_c_start} is equivalent to a full exceptional collection obtained by the blow-up formula from a del Pezzo surface $X'$. Now \cref{lem:exc_obj_del_pezzo} together with \cref{thm:main_result} implies that we can assume that the exceptional collection on $X'$ comes from iterated blow-ups of a copy of $\bbP^2$. After repeating the computations as above for every $(-1)$-curve in the toric system from $\langle \mcO_X(D_1), \mcO_X(D_3), \dots, \mcO_X(D_n) \rangle$ we can assume that \cref{eq:f_e_c_2} is equivalent to a collection
	\begin{equation*}
		\langle \mcO_{E_1'}(-1), \mcO_{E_2'}(-d_2), \dots, \mcO_{E_9'}(-d_9), \mcO_X(nH'), \mcO_X((n+1)H'), \mcO_X((n+2)H') \rangle,
	\end{equation*}
	where $E_1', \dots, E_9'$ are pairwise disjoint $(-1)$-curves with $E_i'H'=0$ and $H'^2=1$. Twisting the partial sequence $$\langle \mcO_{E_2'}(-d_2), \dots, \mcO_{E_9'}(-d_9), \mcO_X(nH'), \mcO_X((n+1)H'), \mcO_X((n+2)H')\rangle$$ with $K_{X'}$ can be realized as a sequence of mutations, because $(-\otimes \mcO_X(K_{X'})[2])$ is the Serre functor of 
	\begin{equation*}	\label{eq:f_e_c_3}
		\langle \mcO_{E_2'}(-d_2), \dots, \mcO_{E_9'}(-d_9), \mcO_X(nH'), \mcO_X((n+1)H'), \mcO_X((n+2)H')\rangle\cong \sfD^b(X').
	\end{equation*}
	Thus we can assume $d_2=1$ and repeating this procedure, we can assume that $d_i=1$ for all $i$. 
	We have an equivalence $\langle \mcO_X(nH'), \mcO_X((n+1)H'), \mcO_X((n+2)H') \rangle \cong \sfD^b(\bbP^2)$, where $H'$ is identified with a hyperplane class. On $\bbP^2$ we compute that $\langle \mcO_{\bbP^2}, \mcO_{\bbP^2}(H), \mcO_{\bbP^2}(2H) \rangle$ is equivalent to $$\langle \mcO_{\bbP^2}(H), \mcO_{\bbP^2}(2H) ,\mcO_{\bbP^2}(-K_{\bbP^2})=\mcO_{\bbP^2}(3H) \rangle,$$ thus in our situation we can assume that $n=0$.
	Therefore $E_1', \dots, E_9', H'$ can be obtained from $E_1, \dots, E_9, H$ by applying an orthogonal transformation of $\Pic(X)$ fixing the canonical class $-3H+\sum_iE_i=-3H'+\sum_iE_i'$. It remains to show that the two sequences
	\begin{align*}
		\sfD^b(X)={ }&\langle \mcO_X, \mcO_X(E_1), \dots, \mcO_X(E_9), \mcO_X(H), \mcO_X(2H) \rangle \; \text{and} \\
		\sfD^b(X)={ }&\langle \mcO_X, \mcO_X(E_1'), \dots, \mcO_X(E_9'), \mcO_X(H'), \mcO_X(2H') \rangle
	\end{align*}
	are equivalent.
	As explained in \cref{rmk:weyl_group}, by \cite[Thm.~0.1]{harbourne_blowings_up_of_p2_and_their_blowings_down} and \cref{lem:left_orth_divisor_-1} the group $\Orth(\Pic(X))_{K_X}$ coincides with the Weyl group generated by the reflections induced by the simple roots $E_1-E_2, \dots, E_8 - E_9$, and $H-E_1-E_2-E_3$.
	The reflection along the hyperplane orthogonal to the a $(-2)$-class $v$ is given by $$s_v(x)=x+(x\cdot v)v.$$
	Thus if $v=E_i-E_{i+1}$, then $s_v$ fixes $E_1, \dots E_{i-1}, E_{i+2}, \dots, H$ and permutes $E_i$ and $E_{i+1}$. This can be identified with a mutation of the exceptional pair $\langle \mcO_X(E_i), \mcO_X(E_{i+1}) \rangle$. Assume $v=H-E_1-E_2-E_3$; then $s_v$ fixes $E_4, \dots, E_9$. Computing the corresponding mutation (on the blow-up of $3$ points for simplicity) one observes that the full exceptional collection $$\sfD^b(X)=\langle \mcO_X, \mcO_X(E_1), \mcO_X(E_2), \mcO_X(E_3), \mcO_X(H), \mcO_X(2H) \rangle$$ is changed to
	\begin{align*}
		\sfD^b(X) =\langle &\mcO_X, \mcO_X(H-E_2-E_3), \mcO_X(H-E_1-E_3), \mcO_X(H-E_1-E_2), \\
		&\mcO_X(2H-E_1-E_2-E_3), \mcO_X(4H-2E_1-2E_2-2E_3) \rangle.
	\end{align*}
	This is the full exceptional collection obtained by the blow-up formula after blowing down the strict transforms of the lines through 2 of the points. By \Cref{lem:two_blow_up_realizations_are_related_by_mutations} this simple reflection can also be realized as a sequence of mutations and shifts.
	In general an element of the Weyl group is a composition of simple reflections $s_{v_1}\circ \dots \circ s_{v_n}$. Recall that for reflections $s_v\circ s_w \circ s_v=s_{s_v(w)}$ holds. This gives $$s_{s_v(w)} \circ s_v= s_v\circ s_w.$$
	Applying this to our composition of simple reflections we can rewrite 
	$$s_{v_1}\circ \dots \circ s_{v_n}=s_{s_{v_1}(v_2)}\circ \dots \circ s_{s_{v_1}(v_n)} \circ s_{v_1}.$$
	We conclude now by induction: After realizing $s_{v_1}$ by mutations and shifts, $s_{s_{v_1}(v_2)}\circ \dots \circ s_{s_{v_1}(v_n)}$ is a sequence of $n-1$ simple reflections with respect to the new basis of simple roots obtained after applying $s_{v_1}$. Hence it can be realized as a sequence of mutations and shifts.
\end{proof}
As a corollary we obtain a new proof of a result of Kuleshov--Orlov.
\begin{corollary}[{cf.~\cite[Thm.~7.7]{kuleshov_orlov_exceptional_sheaves_on_del_pezzo_surfaces}}]\label{cor:new_proof_trans_del_pezzo}
	Let $X$ be a del Pezzo surface, then any two full exceptional collections on $X$ are related by mutations and shifts.
\end{corollary}
\begin{proof}
	Recall that $X$ is either $\bbP^1 \times \bbP^1$ or a blow-up of less than $8$ points in $\bbP^2$ in general position.
	Given the latter case, suppose $E_\bullet$ and $F_\bullet$ are two full exceptional collections on $X$. By \cref{thm:main_result_introduction} we can assume that $E_\bullet$ and $F_\bullet$ consist of rank $1$ objects. Now by \cref{lem:exc_obj_del_pezzo} exceptional rank $1$ objects on $X$ are line bundles and we argue as in the proof of \cref{thm:transitivity_on_9_points}.
	
	Assume $X=\bbP^1\times \bbP^1$, then $\Pic(X)=\bbZ H_1 \oplus \bbZ H_2$ with $H_1H_2=1$ and $H_1^2=H_2^2=0$ and $K_X=-2H_1-2H_2$.
	One computes that the orthogonal transformations of $\Pic(X)$ fixing $K_X$ are exactly the permutations of $H_1$ and $H_2$. Let $E_\bullet$ be a full exceptional collection on $X$. As before we can assume that $E_\bullet$ is a sequence consisting of line bundles. By \cref{thm:main_result_introduction} and \cref{lem:orth_trans_lift_to_isometries} $E_\bullet$ has the form
	\begin{align*}
		&\langle \mcO_X(a,b), \mcO_X(a+1,b), \mcO_X(a,b+1), \mcO_X(a+1,b+1) \rangle \; \text{or}\\
		&\langle \mcO_X(a,b), \mcO_X(a,b+1), \mcO_X(a+1,b), \mcO_X(a+1,b+1) \rangle.
	\end{align*}
	Both are equivalent as the mutation $\sfL_{2,3}$ permutes the middle factors. One computes that the right mutation $\sfR_{\mcO_X(a+1,b+1)}(\mcO_X(a,b+1))$ is equal to $\mcO_X(a+2,b+1)$ up to possible shifts and similarly $\sfR_{\mcO_X(a+1,b+1)}(\mcO_X(a+1,b))$ identifies with $\mcO_X(a+1,b+2)$. We deduce that $E_\bullet$ is equivalent to
	\begin{equation*}
		\langle \mcO_X(a,(b+1)),  \mcO_X(a+1,(b+1)), \mcO_X(a,(b+1)+1), \mcO_X(a+1,(b+1)+1) \rangle,
	\end{equation*}
	hence we realized the twist by $\mcO_X(0,1)$ as a sequence of mutations. Analogously one obtains that the twist by $\mcO_X(1,0)$ is a sequence of mutations and therefore $E_\bullet$ is equivalent to
	\begin{equation*}
		\langle \mcO_X(0,0), \mcO_X(1,0), \mcO_X(0,1), \mcO_X(1,1) \rangle. \qedhere
	\end{equation*}
\end{proof}
\begin{corollary}\label{cor:realizing_isometries_as_mutations}
	Let $X$ be a smooth projective surface over a field $k$ with $\chi(\mcO_X)=1$ and $\rk \sfK_0^{\mathrm{num}}(X) \leq 12$. Then any two exceptional bases $e_\bullet$ and $f_\bullet$ of $\sfK_0^{\mathrm{num}}(X)$ are related by a sequence of mutations and sign changes.
\end{corollary}
\begin{proof}
	By Vial's classification, see \cref{thm:classification}, we can assume that $X$ is a del Pezzo surface or the blow-up of $\bbP^2$ in $9$ points. In these cases $\sfK_0^{\mathrm{num}}(X)$ is independent from the base field and the position of points, thus we can assume that the base field is $\bbC$ and the blown up points are in very general position. Moreover, by Perling's algorithm, see \cref{thm:perlings_algorithm}, we can assume that $e_\bullet$ and $f_\bullet$ only consist of rank $1$ objects. Recall that for a numerically exceptional object $E \in \sfD^b(X)$ the Riemann--Roch formula implies $\cc_2(E)=0$, thus we may assume that $e_\bullet$ and $f_\bullet$ arise from two numerically exceptional collections of maximal length consisting of line bundles. Hence, the corollary follows from \cref{thm:transitivity_on_9_points} and \cref{cor:new_proof_trans_del_pezzo}.
\end{proof}

\section{Blow-up of 10 Points}\label{sec:blow_up_10_points}
Although the situation of $9$ blown up points is similar to the case of del Pezzo surfaces, the situation changes if we blow up $10$ points.
In fact, the conclusions of \cref{lem:left_orth_divisor_-1} and \cref{lem:weyl_group_is_stabilizer} do not hold for the blow-up of $10$ points.
\begin{lemma}\label{lem:weil_group_not_stabilizer_10_points}
	Let $X$ be the blow-up of $\bbP^2$ in $10$ points. Then the stabilizer of the canonical class is
	$$\Orth(\Pic(X))_{K_X}=W_X \times \langle \iota \rangle,$$
	where $W_X$ is the reflection group generated by the simple reflections corresponding to the roots $H-E_1-E_2-E_3, E_1-E_2, \dots, E_9-E_{10}$ and $\iota$ is the involution of $\Pic(X)$ fixing $K_X$ and given by multiplication of $-1$ on $K_X^\perp$.
\end{lemma}
\begin{proof}
	Denote the roots by $\alpha_0\coloneqq H-E_1-E_2-E_3, \alpha_1\coloneqq  E_1-E_2, \dots, \alpha_9\coloneqq E_9-E_{10}$.
	Since $K_X^2=-1$, $\Pic(X)$ splits as an orthogonal direct sum 
	$$\Pic(X) = K_X^\perp \oplus \bbZ K_X.$$
	One can compute that a basis of $K_X^\perp$ is given by the roots $\alpha_0, \dots, \alpha_{9}$. This shows that $K_X^\perp$ is an even unimodular lattice of signature $(1, 9)$. It is known that $\mathrm{II}_{9,1}(-1)$ is the unique even unimodular lattice of signature $(1,9)$. Its orthogonal group was computed by Vinberg \cite{vinberg_some_artithmetical_discrete_groups_in_lobavevski_spaces}. Vinberg's result was rewritten by Conway--Sloane which use the description of $\mathrm{II}_{9,1}$ as the set
	$$\left \{x=(x_0, \dots, x_9) \in \bbZ^{10} \cup (\bbZ+1/2)^{10} \mid x_0 +\dots +x_8-x_9 \in 2\bbZ \right \} \subseteq \bbQ^{10} $$
	with bilinear form $(x, y)\coloneqq \sum_{i=0}^{8}x_iy_i -x_9 y_9$.
	Now \cite[\S\,27 Thm.~1]{conway_sloane_sphere_packings_lattice_and_groups} states that $\Orth(\mathrm{II}_{9,1})=W_{\mathrm{II}_{9,1}} \times \{\pm \id_{\mathrm{II}_{9,1}}\}$, where $W_{\mathrm{II}_{9,1}}$ is the Weyl group of the root system in $\mathrm{II}_{9,1}$ with simple roots
	\begin{equation*}
		\beta_i=(\underbrace{0, \dots, 0}_{i}, 1,-1, \underbrace{0, \dots, 0}_{8-i}) \; \text{for} \; 0 \leq i \leq 7, \beta_8=(1/2, \dots, 1/2), \; \text{and} \; \beta_9=(-1, -1, \underbrace{0, \dots, 0}_{8}).
	\end{equation*}	
	We observe that sending $\alpha_0 \mapsto \beta_0$, $\alpha_i \mapsto \beta_{i-2}$ for $3 \leq i \leq 9$, and $\alpha_i \mapsto \beta_{i+7}$ for $i=1,2$ yields a suitable isomorphism of lattices $K_X^\perp \xrightarrow{\sim}\mathrm{II}_{9,1}(-1)$ such that the $\alpha_i$ are send to the simple roots $\beta_i$. Clearly $\Orth(\Pic(X))_{K_X} = \Orth(K_X^\perp)$, thus the lemma follows.
\end{proof}
A further computation shows that
\begin{equation*}
	D_i  \coloneqq \iota(E_i)=-6H + 2\sum_{j=1}^{10} E_j - E_i \quad \text{and} \quad
	F  \coloneqq \iota(H)=-19H+6\sum_{i=1}^{10}E_i.
\end{equation*}
Thus
\begin{equation*}
	\mcO_X, \mcO_X(D_1) , \dots, \mcO_X(D_{10}), \mcO_X(F), \mcO_X(2F)
\end{equation*}
is a numerically exceptional collection of maximal length on $X$. We show in \cite{krah_upcoming_work} that the collection is exceptional but not full.
The divisors $D_i$ are not effective but satisfy $D_i^2=-1$ and $\chi(D)=1$. This shows that the conclusion of \cref{lem:left_orth_divisor_-1} does not hold for blow-ups of $10$ or more points.
\begin{proposition}
	Let $X$ be the blow-up of $\bbP^2$ in $10$ points in general position. Then $\bbZ^{13} \rtimes \mathfrak{B}_{13}$ does not act transitively on the set of exceptional collections of length $13$.
\end{proposition}
\begin{proof}
	Mutations and shifts do not change the generated subcategory of an exceptional collection. Thus the existence of a full and of a non-full exceptional collection of the same length shows that the action cannot be transitive.
\end{proof}
\begin{proposition}\label{prop:max_2_numerical_oribts_10_points}
	Let $X$ be a smooth projective surface over a field $k$ with $\chi(\mcO_X)=1$ and $\rk \sfK_0^\mathrm{num}(X) = 13$ such that $\sfK_0^\mathrm{num}(X)$ admits an exceptional basis. Then the action of $\{\pm 1\}^{13}\rtimes \mathfrak{B}_{13}$ has at most $2$ orbits.
\end{proposition}
\begin{proof}
	Without loss of generality, we assume that $X$ is the blow-up of $\bbP^2$ in $10$ points in general position. Applying \cref{thm:main_result}, we know that each orbit contains an exceptional basis of the form
	$$ ([\mcO_X(D_1)], \dots, [\mcO_X(D_{13})]),$$
	such that $D_2-D_1 = \varphi(A_1), D_3-D_2 = \varphi(A_2), \dots, D_{13}-D_{12}=\varphi(A_{12}),$
	where $(A_1, \dots, A_{13})$ is the toric system associated to the collection
	$$ \sfD^b(X) = \langle \mcO_X, \mcO_X(E_1), \dots, \mcO_X(E_{10}), \mcO_X(H), \mcO_X(2H) \rangle$$
	and $\varphi \in \Orth(\Pic(X))_{K_X}$. By \cref{lem:weil_group_not_stabilizer_10_points}, either $\varphi \in W_X$ or $\varphi$ can be written as $\iota \circ w$ for some $w \in W_X$.
	Thus, it is enough to show that for each $\varphi \in W_X$ the collections
	\begin{equation*}
		([\mcO_X(D_1)], \dots, [\mcO_X(D_{13})]) \quad \text{and} \quad ([\mcO_X], [\mcO_X(E_1)], \dots, [\mcO_X(E_{10})], [\mcO_X(H)], [\mcO_X(2H)])
	\end{equation*}
	lie in the same orbit.
	As $\varphi \in W_X$ sends $(-1)$-curves to $(-1)$-curves, $D_2-D_1, \dots, D_{11}-D_1$ is a set of disjoint $(-1)$-curves.
	Thus, we can argue as in \cref{thm:transitivity_on_9_points} to reduce to showing that
	\begin{align*}
	&([\mcO_X], [\mcO_X(E_1)], \dots, [\mcO_X(E_{10})], [\mcO_X(H)], [\mcO_X(2H)])\\
		\text{and}\; &([\mcO_X], [\mcO_X(\varphi(E_1))], \dots, [\mcO_X(\varphi(E_{10}))], [\mcO_X(\varphi(H))], [\mcO_X(\varphi(2H))])
	\end{align*}
	lie in the same orbit. But as $\varphi$ can be factored in a sequence of simple reflections, this follows with the same argument as in \cref{thm:transitivity_on_9_points}.
\end{proof}
\begin{remark}\label{rmk:sharpness_of_bound_on_rank}
	An characterization similar to \cref{lem:exc_obj_del_pezzo} of exceptional objects in $\sfD^b(X)$, where $X$ is the blow-up of $\bbP^2_\bbC$ in $10$ points in general position, would be of particular interest. More precisely, if one could verify condition \labelcref{item:condition_ii} from \cref{se:introduction} for exceptional collections of maximal length on $X$, one could conclude that there are $2$ orbits of the $\{\pm 1\}^{13}\rtimes \mathfrak{B}_{13}$-action on exceptional bases of $\sfK_0^\mathrm{num}(X)$. One orbit would consist of the images of full exceptional collections and the the other orbit of the images of exceptional collections of length $13$ which are not full.
\end{remark}

\printbibliography[title=References]

\end{document}